\documentclass{article}

\usepackage{lineno}
\modulolinenumbers[5]

\usepackage{amsmath}
\usepackage{amsfonts}
\usepackage{amssymb}
\usepackage{amsthm}
\usepackage{color}
\usepackage{graphicx}
\usepackage{bm}
\usepackage{mathrsfs}
\RequirePackage[colorlinks,citecolor=blue,urlcolor=blue]{hyperref}

\usepackage{comment}

\usepackage{varwidth}
\usepackage{caption}
\usepackage{tabularx}
\usepackage[export]{adjustbox}
\usepackage{multirow}
\usepackage{booktabs}

\theoremstyle{definition}
\newtheorem{assumptionalt}{Assumption}

\numberwithin{equation}{section}

\theoremstyle{plain}
\newtheorem{lemma}{Lemma}[section]
\newtheorem{proposition}[lemma]{Proposition}
\newtheorem{theorem}[lemma]{Theorem}
\newtheorem{corollary}[lemma]{Corollary}

\theoremstyle{definition}
\newtheorem{definition}[lemma]{Definition}

\newenvironment{assumption}[1]{\renewcommand\theassumptionalt{#1}\assumptionalt}{\endassumptionalt}

\newtheorem{remark}[lemma]{Remark}
\newtheorem{example}[lemma]{Example}

\DeclareMathOperator{\supp}{supp}

\DeclareMathOperator{\Var}{Var}

\DeclareMathOperator{\E}{{\mathbb E}}

\DeclareMathOperator{\R}{{\mathbb R}}

\DeclareMathOperator{\N}{{\mathbb N}}

\DeclareMathOperator{\PP}{{\mathbb P}}

\providecommand{\eps}{\varepsilon}
\renewcommand{\phi}{\varphi}

\renewcommand{\theta}{\vartheta}
\renewcommand{\subset}{\subseteq}

\providecommand{\abs}[1]{\lvert #1 \rvert}
\providecommand{\norm}[1]{\lVert #1 \rVert}

\providecommand{\babs}[1]{{\Bigl\lvert #1 \Bigr\rvert}}
\providecommand{\scapro}[2]{\langle #1,#2 \rangle}

\let\scr\mathscr     
\begin{document}


\title{Parameter estimation for the stochastic \\ heat equation with multiplicative noise\\ from local measurements}

\author{Josef Jan\'ak\\ Karlsruhe Institut f\"ur Technologie,  Germany\\ Email: josefjanak@seznam.cz \and Markus Rei\ss\\ Humboldt-Universit\"at zu Berlin, Germany\\Email: mreiss@math.hu-berlin.de}

\maketitle

\begin{abstract}
For the stochastic heat equation with multiplicative noise we consider the problem of estimating the diffusivity parameter in front of the Laplace operator. Based on local observations in space, we first study an estimator, derived in \cite{AR} for additive noise. A stable central limit theorem shows that this estimator is consistent and asymptotically mixed normal. By taking into account the quadratic variation, we propose two new estimators. Their limiting distributions exhibit a smaller (conditional) variance and the last estimator also works for vanishing noise levels. The proofs are based on local approximation results to overcome the intricate nonlinearities and on a stable central limit theorem for stochastic integrals with respect to cylindrical Brownian motion. Simulation results illustrate the theoretical findings.
\end{abstract}




\section{Introduction}

We consider estimation of the diffusivity parameter $\theta>0$ in the stochastic heat equation with multiplicative noise
\begin{equation} \label{equation SPDE}
\begin{cases}
dX(t) &= \vartheta \Delta X(t) \, dt + \sigma (X(t)) \, dW(t), \quad 0 < t \leq T, \\
X(0) &= X_0, \\
X(t)|_{\partial \Lambda} &= 0, \quad 0 < t \leq T.
\end{cases}
\end{equation}
Here, $W$ is a cylindrical Brownian motion with values in $L^2(\Lambda)$, where $\Lambda$ is an open bounded interval in $\R$, $dW(t)/dt$ is also referred to as space-time white noise. The function $\sigma:\R\to\R_+$ generates a multiplicative noise, see Section \ref{sec:model} below for precise assumptions. Multiplicative noise appears naturally in sto\-chas\-tic partial differential equations (SPDEs) as a scaling limit or to ensure positivity of the solution, see e.g. the examples given in \cite{MP}, \cite{PT} or \cite{BHN}.

Diffusivity estimation has emerged as a benchmark inference problem for SPDEs. The spectral estimation approach, initiated by \cite{HR}, has been shown to give reliable estimation results even for more general semi-linear equations like the stochastic Navier-Stokes equation \cite{CG}, yet always assuming additive noise. In \cite{CL} a specific case of multiplicative noise has been treated which leads to geometric Brownian motions in the spectral decomposition of the Laplacian. In \cite{CCG} also Bayesian estimators have been developed and analysed in this setting.

Similarly, discrete observations of the solution $X$ in time and space give rise to realised $p$-variation estimators for quite general classes of SPDEs. Most notably, in \cite{PT} a precise convergence analysis of $p$-variation of $X(t,x)$ in space $x$ with $p=2$ and in time $t$ with $p=4$ is given, which leads to a consistent diffusivity estimator in the multiplicative noise case, while convergence rates or asymptotic normality are not considered. Estimation of the multiplicative noise function $\sigma(\cdot)$ from discrete observations is treated in \cite{C} with intriguing phenomena arising in central limit theorems for $p$-variations.

Recently, methods for local observations in space  have provided a new methodology for linear and semi-linear  SPDEs with additive noise \cite{AR,ACP}. This has enabled the estimation of diffusivity in a stochastic cell motility model from experimental data \cite{ABJR}. The underlying SPDE with additive noise describes chemical concentrations, for which, however, a multiplicative noise structure might be more natural as well as more in line with the empirical data than additive noise.

Starting point of our work is the question whether the additive noise estimator (ANE) derived in \cite{AR} for local observations of a stochastic heat equation with additive noise is robust against a multiplicative noise misspecification the same way, as it is against nonlinear reaction terms \cite{ACP}. Technically, we cannot use a splitting technique to separate the nonlinear from the linear part and we must derive new  tools to analyse the estimation error. This is achieved by a stepwise disentanglement and localisation of the statistics, carried out in Proposition~\ref{prop:K} below. The result is that the estimator has the same rate as for additive noise, but it is asymptotically mixed normal under stable convergence with a suboptimal conditional variance for varying $\sigma(\cdot)$.

Therefore we improve the ANE by taking into account the varying quadratic variation of the martingale term in the ANE. The multiplicative noise estimator (MNE) obtained this way satisfies a central limit theorem with smaller variance provided the multiplicative noise $\sigma(\cdot)$ is bounded away from zero. Since in many cases it is natural that $\sigma(\cdot)$ vanishes at some boundary values, we improve the MNE to the stabilised multiplicative noise estimator (SMNE), satisfying a stable central limit theorem with small conditional variance  even when $\sigma(\cdot)$ vanishes sometimes.

The exact setting is introduced in Section \ref{sec:model}. The construction of the estimators, the main asymptotic results and an application to confidence intervals are presented in Section \ref{sec:mainresults}. In Section \ref{sec:simulation} we discuss the implementation of the estimators and their behaviour for three fundamentally different  noise specifications. The detailed proofs are delegated to Section \ref{sec:proofs}. The stable convergence results require a martingale representation theorem in terms of cylindrical Brownian motion and rely on asymptotic orthogonality of martingales by spatial localisation, which might be of independent interest. This material is therefore gathered in Section \ref{sec:stableCLT}.

\section{The model}\label{sec:model}

\subsection{Notations}

We write $\mathbb R_+ := [0, \infty)$, $a \wedge b := \min (a,b)$ and $a \vee b := \max (a,b)$. By $A_\delta \lesssim B_\delta$ we mean that there exists some constant $C > 0$ such that $A_\delta \leq C B_\delta$ for all values $\delta$ under consideration. Here, we work with $\delta\in(0,1)$ or with the convergence $\delta\to 0$. Convergence in probability and convergence in distribution are denoted by $\stackrel{\mathbb P}{\rightarrow}$ and $\stackrel{d}{\rightarrow}$, respectively. The symbol $\xrightarrow{stably}$ denotes  stable convergence, see e.g. \cite{JS}, Chapter VIII. Section 5. We say that $X_\delta\xrightarrow{stably} X$ holds on an event $G$ if $X_\delta{\bf 1}_G \xrightarrow{stably} X{\bf 1}_G$. The symbol $A_\delta = O_{\mathbb P}(B_\delta)$ for random variables $A_\delta,B_\delta$  means that $A_\delta / B_\delta$ is tight, that is, $\sup_\delta \mathbb P (|A_\delta| > C |B_\delta|) \rightarrow 0$ as $C \rightarrow \infty$. The notation $A_\delta = o_{\mathbb P}(B_\delta)$ stands for $A_\delta / B_\delta \stackrel{\mathbb P}{\rightarrow} 0$ as $\delta\to 0$.

Let $\Lambda$ be an open bounded interval in $\mathbb R$ and consider the space $L^2(\Lambda)$ equipped with the usual $L^2$-norm $\| \cdot \| := \| \cdot \|_{L^2(\Lambda)}$ and the scalar product $\left\langle \cdot, \cdot \right\rangle := \left\langle \cdot, \cdot \right\rangle_{L^2(\Lambda)}$. The space $C_b(\R)$ of all continuous bounded functions on $\R$ is equipped with the supremum norm $\| \cdot \|_{\infty}$. $H^s(\Lambda)$ denotes the $L^2$-Sobolev space of order $s$ on $\Lambda$ and $H_0^1(\Lambda)$ the space of all $f\in H^1(\Lambda)$ with $f(x)=0$ for $x\in\partial\Lambda$. We use the standard Laplace operator notation $\Delta z = z''$ for $z\in H^2(\R)$, even in the simple one-dimensional case.

\subsection{The stochastic heat equation}

Let $(\Omega, \scr F, (\scr F_t)_{0 \leq t \leq T}, \mathbb P)$ be a stochastic basis equipped with the cylindrical Brownian motion $W$ taking values in $L^2(\Lambda)$. The filtration $(\scr F_t)_{0 \leq t \leq T}$ is assumed to be generated by the cylindrical Brownian motion and augmented by $\mathbb P$--null sets.
We study the stochastic heat equation \eqref{equation SPDE} with  multiplicative noise.
The initial value $X_0 \in L^2(\Lambda)$ is supposed to be deterministic and continuous on $\bar\Lambda$. We require throughout the following two assumptions.

\begin{assumption}{(S)}\label{ass:sigma}
The function $\sigma : \mathbb R \rightarrow \mathbb R_+$ is continuous.
\end{assumption}


The stochastic term $\sigma (X(t)) \, dW(t)$ is therefore understood as $B(X(t))dW(t)$ with the multiplicative Nemytskii operators $B(u): L^2(\Lambda) \rightarrow L^2(\Lambda)$, defined by
$$
\left( B(u) v \right)(x) = \sigma (u(x)) v(x), \quad x \in \Lambda, \, u\in C(\bar\Lambda),\, v \in L^2(\Lambda),
$$
noting that $\sigma(u(\cdot))\in L^\infty(\Lambda)$ holds for continuous and thus bounded functions $u\in C(\bar\Lambda)$.

\begin{assumption}{(X)}\label{ass:X}
The stochastic partial differential equation \eqref{equation SPDE} admits a unique weak solution $(X(t), 0 \leq t \leq T)$ taking values in $L^2(\Lambda)$, which  is continuous in both (time and space) variables, i.e., $X \in C([0,T]; C(\bar\Lambda; \mathbb R))$ $\mathbb P$--a.s., and satisfies for any $z \in H_0^1(\Lambda) \cap H^2(\Lambda)$ and $t \in [0,T]$
\begin{equation}
\left\langle X(t), z \right\rangle = \left\langle X_0, z \right\rangle + \vartheta \int_0^t \left\langle X(s), \Delta z \right\rangle \, ds + \int_0^t \left\langle z, \sigma(X(s)) \, dW(s) \right\rangle, \quad \mathbb P-a.s.
\end{equation}
\end{assumption}

Sufficient conditions for Assumption \ref{ass:X}  will be discussed  in Example \ref{ex:XSprime} below.
In our setup a weak solution is also a mild solution (see e.g. Theorem 6.5 in \cite{DPZ}). If $(S_{\vartheta}(t), t \geq 0)$ is the strongly continuous semigroup on $L^2(\Lambda)$ generated by $\vartheta \Delta$, then the solution $(X(t), 0 \leq t \leq T)$ to equation \eqref{equation SPDE} is given by the variation-of-constants formula
\begin{equation} \label{mild solution}
X(t) = S_{\vartheta}(t) X_0 + \int_0^t S_{\vartheta}(t-s) \sigma(X(s)) \, dW(s), \quad \mathbb P-a.s.
\end{equation}

\subsection{The observation scheme}

As motivated in \cite{AR,ABJR}, we observe the solution process $(X(t,x),\,t\in[0,T],x\in\Lambda)$ only locally in space around some point $x_0\in\Lambda$. That point $x_0$ as well as the terminal time $T\in(0,\infty)$ remain fixed.

More precisely, the observations are given by a spatial convolution of the solution process with a kernel $K_{\delta,x_0}$, localising at $x_0$ as the resolution $\delta$ tends to zero. This kernel might for instance model the {\it point spread function} in microscopy.

For $z \in L^2(\mathbb R)$ and $\delta \in(0,1)$ introduce the scalings
\begin{align*}
\Lambda_{\delta, x_0} &:= \delta^{-1} \left( \Lambda - x_0 \right) = \{ \delta^{-1} (x - x_0) : x \in \Lambda \}, \\
z_{\delta, x_0}(x) &:= \delta^{-1/2} z\left( \delta^{-1} (x - x_0) \right), \quad x \in \mathbb R.
\end{align*}
Throughout this paper, $K \in H^2(\mathbb R)$ denotes a fixed function of compact support in $\Lambda_{1, x_0}$, called kernel. The compact support ensures that $K_{\delta, x_0}$ is localising around $x_0$ as $\delta\to 0$ and that $K_{\delta, x_0} \in H_0^1(\Lambda) \cap H^2(\Lambda)$. The scaling with $\delta^{-1/2}$ simplifies calculations due to $\| K_{\delta, x_0} \| = \| K \|_{L^2(\mathbb R)}$, while  the basic estimators are invariant with respect to kernel scaling.

Local measurements of \eqref{equation SPDE} at the point $x_0$ with  resolution level $\delta\in(0,1)$ are described by the real-valued processes $(X_{\delta, x_0}(t), 0 \leq t \leq T)$ and $(X_{\delta, x_0}^{\Delta}(t), 0 \leq t \leq T)$ given by
\begin{align}
X_{\delta, x_0}(t) &= \left\langle X(t), K_{\delta, x_0} \right\rangle, \label{X-delta} \\
X_{\delta, x_0}^{\Delta}(t) &= \left\langle X(t), \Delta K_{\delta, x_0} \label{X-delta-delta} \right\rangle.
\end{align}
 The process $(X_{\delta, x_0}(t), 0 \leq t \leq T)$ satisfies $X_{\delta, x_0} (0) = \left\langle X_0, K_{\delta, x_0} \right\rangle$ and by partial integration
\begin{align}
dX_{\delta, x_0} (t) &= \vartheta X_{\delta,x_0}^\Delta(t) \, dt + \left\langle \sigma (X(t)) K_{\delta, x_0} , \, dW(t) \right\rangle. \label{dX-delta}
\end{align}

\section{Estimation methods and main results} \label{sec:mainresults}

\subsection{The additive noise estimator}

We study first the augmented maximum likelihood estimator  $\hat{\vartheta}_{\delta}$ from \cite{AR}, derived for the stochastic heat equation with additive space-time white noise.

\begin{definition}
The additive noise estimator (ANE) $\hat\vartheta_{\delta}$ of the parameter $\vartheta > 0$ is defined as
\begin{equation}
\hat\vartheta_{\delta} = \frac{\int_0^T X_{\delta, x_0}^{\Delta} (t) \, dX_{\delta, x_0} (t)}{\int_0^T ( X_{\delta, x_0}^{\Delta} (t) )^2 \, dt}.
\end{equation}
\end{definition}

According to \eqref{dX-delta}, the numerator $\int_0^T X_{\delta, x_0}^{\Delta} (t) \, dX_{\delta, x_0} (t)$ equals
\[
  \vartheta \int_0^T \left( X_{\delta, x_0}^{\Delta} (t) \right)^2 \, dt
+ \int_0^T X_{\delta, x_0}^{\Delta} (t) \left\langle \sigma (X(t)) K_{\delta, x_0}, dW(t) \right\rangle
\]
and the fundamental error decomposition  is given by
\begin{equation} \label{eq:errorforhat}
\delta^{-1} ( \hat{\vartheta_{\delta}} - \vartheta ) = \frac{\mathcal M_{\delta}}{\mathcal I_{\delta}^{1/2}} \cdot \frac{\left( \delta^2 \mathcal I_{\delta} \right)^{1/2}}{\delta^2 \mathcal J_{\delta}},
\end{equation}
where
\begin{align*}
\mathcal M_{\delta} &= \int_0^T X_{\delta, x_0}^{\Delta} (t) \left\langle \sigma (X(t)) K_{\delta, x_0}, dW(t) \right\rangle, \\
\mathcal I_{\delta} &= \int_0^T \| \sigma (X(t)) K_{\delta, x_0} \|^2 \left( X_{\delta, x_0}^{\Delta} (t) \right)^2 \, dt, \\
\mathcal J_{\delta} &= \int_0^T \left( X_{\delta, x_0}^{\Delta} (t) \right)^2 \, dt.
\end{align*}
The term $\mathcal I_{\delta}$ is incorporated because it gives the quadratic variation of the martingale $\mathcal M_{\delta}$ in time. It turns out that $\delta^2{\mathcal I_{\delta}}$  converges in probability to the
limit $(2\theta)^{-1}\| K' \|_{L^2(\mathbb R)}^2 \| K \|_{L^2(\mathbb R)}^2 \int_0^T \sigma^4 (X(t, x_0)) \, dt$ as $\delta\to 0$, while
$\delta^2\mathcal J_{\delta}$ converges to $(2\theta)^{-1}\| K' \|_{L^2(\mathbb R)}^2 \int_0^T \sigma^2 (X(t, x_0)) \, dt$, compare Proposition \ref{convergence of IJ_delta} below for bounded $\sigma(\cdot)$. Since the quadratic variation $\mathcal I_{\delta}$ does not become asymptotically deterministic, we cannot rely on a standard martingale central limit theorem (e.g. Theorem 1.19 in \cite{Ku}) to prove asymptotic normality of $\mathcal M_{\delta}/\mathcal I_{\delta}^{1/2}$.

Therefore we employ the concept of stable convergence, which is stronger than convergence in distribution and allows to formulate mixed normal limits and to derive feasible confidence intervals, see e.g. \cite{JS} for a general introduction. In Section \ref{sec:stableCLT} we prove a general martingale representation theorem and a stable limit theorem for martingales with respect to cylindrical Brownian motion filtrations. As a consequence, we obtain the following result, when specialising Corollary \ref{coroll:supp} to our setting involving the kernels $K_{\delta,x_0}$:

\begin{proposition}\label{prop:StableCLT}
Let $(Y_\delta(t), 0\le t\le T)$ for $\delta\in(0,1)$ be  $L^2(\Lambda)$-valued processes, progressively measurable with respect to the cylindrical Brownian filtration $({\scr F}_t)_{0\le t\le T}$ and satisfying $\int_0^T \| Y_\delta(t) \|^2 \, dt < \infty$. If
\begin{itemize}
\item[(C1)] $\int_0^T \| Y_\delta(t) \|^2 \, dt \stackrel{\mathbb P}{\rightarrow} \int_0^T s^2(t) \, dt$ as $\delta \rightarrow 0$ for some progressively measurable real-valued process $(s(t),0\le t\le T)$ with $\int_0^T s^2(t) \, dt<\infty$,
\item[(C2')] the support inclusion $\supp(Y_\delta(t))\subset \supp(K_{\delta,x_0})$ holds Lebesgue-almost everywhere in $\Lambda$ for all $t\in[0,T]$,
\end{itemize}
then a stable limit theorem for the stochastic integrals holds as $\delta\to 0$:
$$
\int_0^T \left\langle Y_\delta(t), dW(t) \right\rangle \xrightarrow{stably}  \int_0^T s(t) \, dB(t)
$$
with an independent scalar Brownian motion $(B(t), 0\le t\le T)$ (on an extension of the original filtered probability space).
\end{proposition}

The main point of this result is that the limiting Brownian motion $B$ becomes independent because the support of $Y_\delta(t)$ shrinks asymptotically to the point $x_0$. Here we shall apply the proposition with $Y_\delta(t)= \delta X_{\delta, x_0}^{\Delta} (t) \sigma (X(t)) K_{\delta, x_0}$. Our first main result is that the additive noise estimator $\hat{\vartheta}_{\delta}$ satisfies a stable central limit theorem with rate $\delta$.

\begin{theorem} \label{thm:hatvartheta}
Grant Assumptions \ref{ass:sigma} and \ref{ass:X}. Then the ANE $\hat{\vartheta}_{\delta}$ satisfies on the event $\{\int_0^T\sigma^2(X(t,x_0))\,dt>0\}$
\begin{equation}\label{EqANECLT}
\delta^{-1} ( \hat{\vartheta}_{\delta} - \vartheta ) \xrightarrow{stably} \frac{(2 \vartheta)^{1/2} \| K \|_{L^2(\mathbb R)}}{\| K' \|_{L^2(\mathbb R)}} \cdot \frac{(\int_0^T \sigma^4 (X(t, x_0)) \, dt )^{1/2}}{\int_0^T \sigma^2 (X(t, x_0)) \, dt} \cdot Z
\end{equation}
as $\delta \rightarrow 0$, where $Z \sim N(0,1)$ is independent of $\scr F_T$.
\end{theorem}

\begin{proof}
The detailed proof is deferred to Section \ref{app:stableconv}.
\end{proof}

This result establishes a very desirable robustness property of the ANE $\hat{\vartheta}_{\delta}$: Even though it was designed for estimation in the stochastic heat equation with additive noise, the ANE still converges with the same rate $\delta$ to the true pa\-ra\-me\-ter under multiplicative noise. This is analogous to the ordinary least squares estimator in linear regression with heteroskedastic noise, which still attains optimal rates, yet loses in the variance due to the variability in noise levels. An efficient regression estimator is  obtained by a noise-level weighted least squares method, which provides an analogy for our next estimators.

\subsection{The multiplicative noise estimator}

We aim at improving the ANE by adjusting the estimator  in such a way that the denominator contains already the quadratic variation of the martingale part in the numerator. To that end, we need to incorporate the term $\| \sigma (X(t)) K_{\delta, x_0} \|^2$ that is not observed directly, but  is still attainable from the data.
The quadratic variation of the observed semi-martingale $(X_{\delta, x_0}(t), 0 \leq t \leq T)$ equals
\begin{equation} \label{eq:QV}
\left\langle X_{\delta, x_0} \right\rangle_t = \int_0^t \| \sigma (X(s)) K_{\delta, x_0} \|^2 \, ds,\quad t\in[0,T].
\end{equation}
So we have access to $\| \sigma (X(t)) K_{\delta, x_0} \|^2$ by differentiation of the realized quadratic variation. For discrete time data, sampled at high-frequency, some spot volatility estimators from the field of mathematical finance can be used to access this term, see Section \ref{sec:simulation} below. This way we obtain a second estimator, taking into account the multiplicative noise in the stochastic heat equation.

\begin{definition} \label{def:tilde estimator}
The multiplicative noise estimator (MNE) $\tilde{\vartheta}_{\delta}$ of the parameter $\vartheta > 0$ is defined as
\begin{equation} \label{eq:tilde estimator}
\tilde{\vartheta}_{\delta} = \frac{\int_0^T \frac{X_{\delta, x_0}^{\Delta} (t)}{\| \sigma(X(t)) K_{\delta, x_0} \|^2} \, dX_{\delta, x_0} (t)}{\int_0^T \frac{( X_{\delta, x_0}^{\Delta} (t) )^2}{\| \sigma(X(t)) K_{\delta, x_0} \|^2} \, dt}.
\end{equation}
\end{definition}

Let us remark that the MNE $\tilde{\vartheta}_{\delta}$ can also be derived like the ANE $\hat{\vartheta}_{\delta}$ in \cite{AR}, maximising a corresponding pseudo-likelihood in the multiplicative noise case. An alternative interpretation is that we regress the increment $dX_{\delta, x_0} (t)$ on $X_{\delta, x_0}^{\Delta} (t)$ and weight it by the inverse squared noise level $\| \sigma(X(t)) K_{\delta, x_0} \|^{-2}$, exactly as in weighted least squares for regression.  Since this is done under the correct model specification, we expect better estimation properties.

Using the representation of $dX_{\delta, x_0} (t)$ from \eqref{dX-delta}, we obtain
\begin{equation} \label{eq:errorfortilde}
\tilde{\vartheta}_{\delta} - \vartheta = \frac{\int_0^T \frac{X_{\delta, x_0}^{\Delta} (t)}{\| \sigma(X(t)) K_{\delta, x_0} \|^2} \, \langle \sigma(X(t)) K_{\delta, x_0}, dW(t) \rangle}{\int_0^T \frac{( X_{\delta, x_0}^{\Delta} (t) )^2}{\| \sigma(X(t)) K_{\delta, x_0} \|^2} \, dt} =: \frac{\tilde{\mathcal M}_{\delta}}{\tilde{\mathcal I}_{\delta}}.
\end{equation}
Since $\| \sigma (X(t)) K_{\delta, x_0} \|^2$ appears in the denominators, we require a lower bound on $\sigma(\cdot)$ in the following theorem.

\begin{theorem} \label{thm:tildevartheta}
Grant Assumptions \ref{ass:sigma}, \ref{ass:X}  and assume $\underline{\sigma}= \inf_{x \in \mathbb R} \sigma(x) > 0$. Then as $\delta \rightarrow 0$
\begin{equation}
\delta^{-1} ( \tilde{\vartheta}_{\delta} - \vartheta ) \stackrel{d}{\rightarrow} N \Big( 0, \frac{2 \vartheta \| K \|_{L^2(\mathbb R)}^2}{T \| K' \|_{L^2(\mathbb R)}^2} \Big).
\end{equation}
\end{theorem}

\begin{proof} The proof is deferred to Section \ref{app:stableconv}.
\end{proof}

\subsection{The stabilised multiplicative noise estimator}

The lower bound $\underline\sigma>0$ on $\sigma(\cdot)$ required for the MNE $\tilde\theta_\delta$ can be restrictive. For instance, when the random field $X(t,x)$ shall not take negative values, models usually require that $\lim_{x\downarrow 0}\sigma(x)=0$. To cover this case as well, we stabilise the denominators in the integrands of equation \eqref{eq:tilde estimator} by  adding a number $\varepsilon_{\delta}^2$ which tends to zero slowly as $\delta \rightarrow 0$.

\begin{definition} \label{def:star estimator}
Let $\varepsilon_{\delta} = \varepsilon(\delta)$ be a real function satisfying for any $\eta>0$
\begin{equation}\label{eq:epsilon}
\varepsilon_{\delta} \rightarrow 0, \quad
\varepsilon_{\delta}^{-1} \delta^{\eta} \rightarrow 0
\end{equation}
as $\delta\to 0$. Then the stabilised multiplicative noise estimator (SMNE) $\vartheta_{\delta}^{\star}$ of the parameter $\vartheta > 0$ is defined as
\begin{equation} \label{eq:star estimator}
\vartheta_{\delta}^{\star} = \frac{\int_0^T \frac{X_{\delta, x_0}^{\Delta} (t)}{\| \sigma(X(t)) K_{\delta, x_0} \|^2 + \varepsilon_{\delta}^2} \, dX_{\delta, x_0} (t)}{\int_0^T \frac{( X_{\delta, x_0}^{\Delta} (t) )^2}{\| \sigma(X(t)) K_{\delta, x_0} \|^2 + \varepsilon_{\delta}^2} \, dt}.
\end{equation}
\end{definition}

Condition \eqref{eq:epsilon}  says that $\eps_\delta$ tends to zero more slowly than any polynomial. It is satisfied for $\varepsilon_{\delta} = \frac{1}{\log(\delta^{-1})}$.
To analyse the asymptotic properties of the SMNE $\vartheta_{\delta}^{\star}$, we need to strengthen Assumptions \ref{ass:sigma} and \ref{ass:X}  slightly.

\begin{assumption}{(S')}\label{ass:sigma'}
The function $\sigma(\cdot)$ is $\beta_{\sigma}$--H\"older continuous, i.e., for some $\beta_{\sigma} \in (0,1]$ there exists a constant $C>0$ such that
$$
\forall x,y\in\Lambda:\;|\sigma(x) - \sigma(y)| \leq C |x-y|^{\beta_{\sigma}}.
$$
\end{assumption}

\begin{assumption}{(X')}\label{ass:X'}
Assumption \ref{ass:X} is satisfied and moreover the solution $X$ is in quadratic mean $\beta_x$--H\"older continuous in the space variable and $\beta_t$--H\"older continuous in the time variable, i.e., for some $\beta_x, \beta_t \in (0,1]$ there is a constant $C>0$ with
\begin{equation}\label{Eq:XHoelder}
\forall t,s\in[0,T],\,x,y\in\Lambda:\;\mathbb E (X(t,x) - X(s,y) )^2 \leq C \left( |t-s|^{2 \beta_t} + |x-y|^{2 \beta_x} \right).
\end{equation}
\end{assumption}

\begin{example} \label{ex:XSprime}
If $\sigma(\cdot)$ is Lipschitz continuous and the initial condition $X_0$ is continuous, then standard contraction arguments for the stochastic convolution and the regularity of the Green function for the heat equation yield Assumption \ref{ass:X}, see e.g. \cite{Ce}. Even Assumption \ref{ass:X'} holds with $\beta_x=1/2$ and $\beta_t=1/4$, provided $\sigma(\cdot)$ is Lipschitz continuous and the initial condition $X_0$ is $1/2$--H\"older continuous. In fact, standard proofs for pathwise H\"older regularity go via the Kolmogorov-Chentsov theorem and thus establish \eqref{Eq:XHoelder}, compare Theorem 2.1 in \cite{PT} or Corollary 3.4 in \cite{W} for a slightly more involved case on an unbounded domain.

The intriguing questions of weak existence, regularity and pathwise uniqueness for the stochastic heat equation with $\beta_\sigma$-H\"older continuous multiplicative noise $\sigma(\cdot)$ have so far only found partial answers. We refer to Theorem 1.3 in \cite{MP}, which yields our Assumption \ref{ass:X} in case $\beta_\sigma>3/4$ in case of an unbounded domain. For their continuity result the authors assert that the results in \cite{SS}, formulated for coloured noise in space, work analogously for the space-time white noise case. Equations (10) and (19) in \cite{SS} then establish H\"older regularity of $X$ in the sense of Assumption \ref{ass:X'}.
\end{example}

We turn to the analysis of the stabilised multiplicative noise estimator.
The error decomposition for the SMNE $\vartheta_{\delta}^{\star}$ follows from \eqref{eq:star estimator} and \eqref{dX-delta}:
\begin{equation} \label{eq:errorforstar}
\delta^{-1} \left(\vartheta_{\delta}^{\star} - \vartheta \right) = \frac{\mathcal M_{\delta}^{\star}}{(\mathcal I_{\delta}^{\star})^{1/2}} \cdot \frac{(\delta^2 \mathcal I_{\delta}^{\star})^{1/2}}{\delta^2 \mathcal J_{\delta}^{\star}}
\end{equation}
with
\begin{align*}
\mathcal M_{\delta}^{\star} &= \int_0^T \frac{X_{\delta, x_0}^{\Delta} (t)}{\| \sigma(X(t)) K_{\delta, x_0} \|^2 + \varepsilon_{\delta}^2} \, \left\langle \sigma(X(t)) K_{\delta, x_0}, dW(t) \right\rangle, \\
\mathcal I_{\delta}^{\star} &= \int_0^T \frac{\| \sigma (X(t)) K_{\delta, x_0} \|^2}{(\| \sigma (X(t)) K_{\delta, x_0} \|^2 + \varepsilon_{\delta}^2)^2} \left( X_{\delta, x_0}^{\Delta} (t) \right)^2 \, dt, \\
\mathcal J_{\delta}^{\star} &= \int_0^T \frac{( X_{\delta, x_0}^{\Delta} (t) )^2}{\| \sigma(X(t)) K_{\delta, x_0} \|^2 + \varepsilon_{\delta}^2} \, dt.
\end{align*}

The term $\mathcal I_{\delta}^{\star}$ is the quadratic variation of the martingale part $\mathcal M_{\delta}^{\star}$.
The limits of $\mathcal I_{\delta}^{\star}$ and $\mathcal J_{\delta}^{\star}$ for $\delta\to 0$ involve a (in general random) time length $T^\star$ during which $\sigma(X(t,x_0))$ does not vanish, compare Proposition \ref{convergence of IJ_delta} below. So, we use again the stable limit theorem of Proposition \ref{prop:StableCLT} and derive a central limit theorem for the SMNE $\vartheta_{\delta}^{\star}$ with rate $\delta$, without assuming a lower bound on $\sigma(\cdot)$.

\begin{theorem} \label{thm:starvartheta}
Grant Assumptions \ref{ass:sigma'} and \ref{ass:X'} with \eqref{eq:epsilon}. Introduce
$$
T^{\star} = \int_0^T \mathbf{1}\left(\sigma(X(t, x_0)) \neq 0 \right) \, dt.
$$
Then as $\delta \rightarrow 0$  on  the event $\{T^\star>0\}$
\begin{equation}
\delta^{-1} \left( \vartheta_{\delta}^{\star} - \vartheta \right) \xrightarrow{stably} \frac{(2 \vartheta)^{1/2} \| K \|_{L^2(\mathbb R)}}{(T^{\star})^{1/2} \| K' \|_{L^2(\mathbb R)}} \cdot Z,
\end{equation}
where $Z \sim N(0,1)$ is independent of $\scr F_T$.
\end{theorem}
\begin{proof}
The proof is deferred to Section \ref{app:stableconv}.
\end{proof}

\begin{remark}
From the series of inequalities
\begin{equation} \label{eq:asympvariances}
\frac{(\int_0^T \sigma^4 (X(t, x_0)) \, dt )^{1/2}}{\int_0^T \sigma^2 (X(t, x_0)) \, dt} \geq \frac{1}{\sqrt{T^\star}} \geq \frac{1}{\sqrt{T}}
\end{equation}
we infer that the (conditional) asymptotic variance of the SMNE lies between those of the ANE and the MNE. Remember, however, that the  asymptotics for the MNE were derived under the condition $\underline\sigma>0$, implying $T^\star=T$. The extreme case $\sigma(\cdot) \equiv 0$ leads to the deterministic heat equation, which for the initial condition $X_0=0$ remains zero all the time and does not allow for inference on $\theta$. This type of degeneracy is excluded for the SMNE by the condition $T^\star>0$.
\end{remark}

\subsection{Confidence intervals}

The asymptotic (mixed) normality of the three estimators allows us to prescribe asymptotic confidence intervals for the parameter $\theta$. The asymptotic conditional variances depend on quantities unknown to the statistician. Yet, in all three error decompositions \eqref{eq:errorforhat}, \eqref{eq:errorfortilde} and \eqref{eq:errorforstar} it is shown in the proofs (see Section \ref{app:stableconv} for the details) that the martingale term divided by the square root of its quadratic variation is asymptotically standard Gaussian. Dividing each error decomposition by the respective second factor on the right-hand side directly gives an asymptotic confidence statement.

\begin{corollary}\label{CorConfI}
Let $\alpha\in(0,1)$. Based on the three estimators $\hat{\vartheta}_{\delta}$, $\tilde{\vartheta}_{\delta}$ and $\vartheta_{\delta}^{\star}$ the  confidence intervals for $\theta$
\begin{align*}
\hat I_{1 - \alpha} &= \Big[ \hat{\vartheta}_{\delta} - \frac{\mathcal I_{\delta}^{1/2}}{\mathcal J_{\delta}} \cdot q_{1 - {\alpha}/2}, \hat{\vartheta}_{\delta} + \frac{\mathcal I_{\delta}^{1/2}}{\mathcal J_{\delta}} \cdot q_{1 - {\alpha}/2} \Big], \\
\tilde{I}_{1 - {\alpha}} &= \Big[ \tilde{\vartheta}_{\delta} - \frac{1}{\tilde{\mathcal I}_{\delta}^{1/2}} \cdot q_{1 - {\alpha}/2}, \tilde{\vartheta}_{\delta} + \frac{1}{\tilde{\mathcal I}_{\delta}^{1/2}} \cdot q_{1 - {\alpha}/2} \Big], \\
I^{\star}_{1 - {\alpha}} &= \Big[ {\vartheta}_{\delta}^\star - \frac{(\mathcal I_{\delta}^\star)^{1/2}}{\mathcal J_{\delta}^\star} \cdot q_{1 - {\alpha}/2}, {\vartheta}_{\delta}^\star + \frac{\mathcal (\mathcal I_{\delta}^\star)^{1/2}}{\mathcal J_{\delta}^\star} \cdot q_{1 - {\alpha}/2} \Big]
\end{align*}
with the standard Gaussian $(1-\alpha/2)$-quantile $q_{1-\alpha/2}$ have each asymptotic coverage $1-\alpha$ as $\delta\to 0$ under the assumptions of Theorems \ref{thm:hatvartheta}, \ref{thm:tildevartheta} and \ref{thm:starvartheta}, respectively.
\end{corollary}

Note that the confidence intervals only rely on the observation processes $(X_{\delta, x_0}^{\Delta}(t), 0 \leq t \leq T)$, $(X_{\delta, x_0}(t), 0 \leq t \leq T)$ and the quadratic variation of the latter. Even the kernel $K$ and the resolution level $\delta$ need not be known. In the next section we shall see how the estimation methods can be implemented when only data is available that is discretely sampled in time.

\section{Implementation and simulation results} \label{sec:simulation}

We illustrate the main results in a  setting similar to the experimental setup in \cite{ABJR}, where  the diffusivity parameter $\vartheta$ was estimated in a concrete stochastic model for cell repolarisation.

Consider the stochastic heat equation \eqref{equation SPDE} with $\Lambda = (0, L)$ for $L=20$, $T=30$, $\vartheta = 0.05$. The initial condition $X_0$ is a smooth approximation of the function $f(x) = 4 \times \mathbf{1}_{[L/4,3L/4]}(x) + 2 \times \mathbf{1}_{(0,L/4) \cup (3L/4,L)}(x)$. We present the results for three different functions $\sigma$:
\begin{align*}
\sigma_1(x) &= 0.20, \\
\sigma_2 (x) &= (0.20 \times |x|^{0.80} + 0.01), \\
\sigma_3 (x) &= 10 \times \exp(-10 \times |x-2|) + 10 \times \exp(-10 \times |x-4|).
\end{align*}
We have chosen $\sigma_2(\cdot)$ to have Hölder regularity 0.8 in line with Example \ref{ex:XSprime} and not to vanish completely at zero so that all three estimators are applicable. $\sigma_3(\cdot)$ generates strong noise level fluctuations so that the quality of the estimators should differ significantly.

An approximate solution is computed on a regular time-space grid $\{ (t_j, y_k): t_j = Tj/N, y_k = Lk/M, j = 0, \ldots, N, k = 0, \ldots, M \}$ with $N = 48 \, 000$ and $M = 800$ by the Euler-Maruyama scheme. For the drift part, we use the finite difference approximation of $\Delta$ that is applied implicitly, while  $\sigma(\cdot)$ in the stochastic term is applied to the current state of the solution explicitly, compare Algorithm 10.8 in \cite{LPS}. The mesh sizes fulfill $T/N \asymp (L/M)^2$, ensuring the Courant-Friedrichs-Lewy (CFL) condition for stable simulations \cite{LPS}. Heat maps for typical realisations with multiplicative noise $\sigma_2(X(t))$ and $\sigma_3(X(t))$ are displayed in Figure \ref{fig:1+2}. Under $\sigma_2(\cdot)$ we see that fluctuations are larger for higher temperature levels, while at the boundary it cools down to zero almost deterministically. Under $\sigma_3(\cdot)$ excitations by strong noise at the interface values 2 and 4 are counteracted by the diffusion, which leads to almost noiseless inner regions with strong fluctuations of the interfaces in time. The spatial gradient at the interfaces is very large, which is no numerical artefact, but due to expulsion by noise.

\begin{figure}
\includegraphics[width=0.5\textwidth]{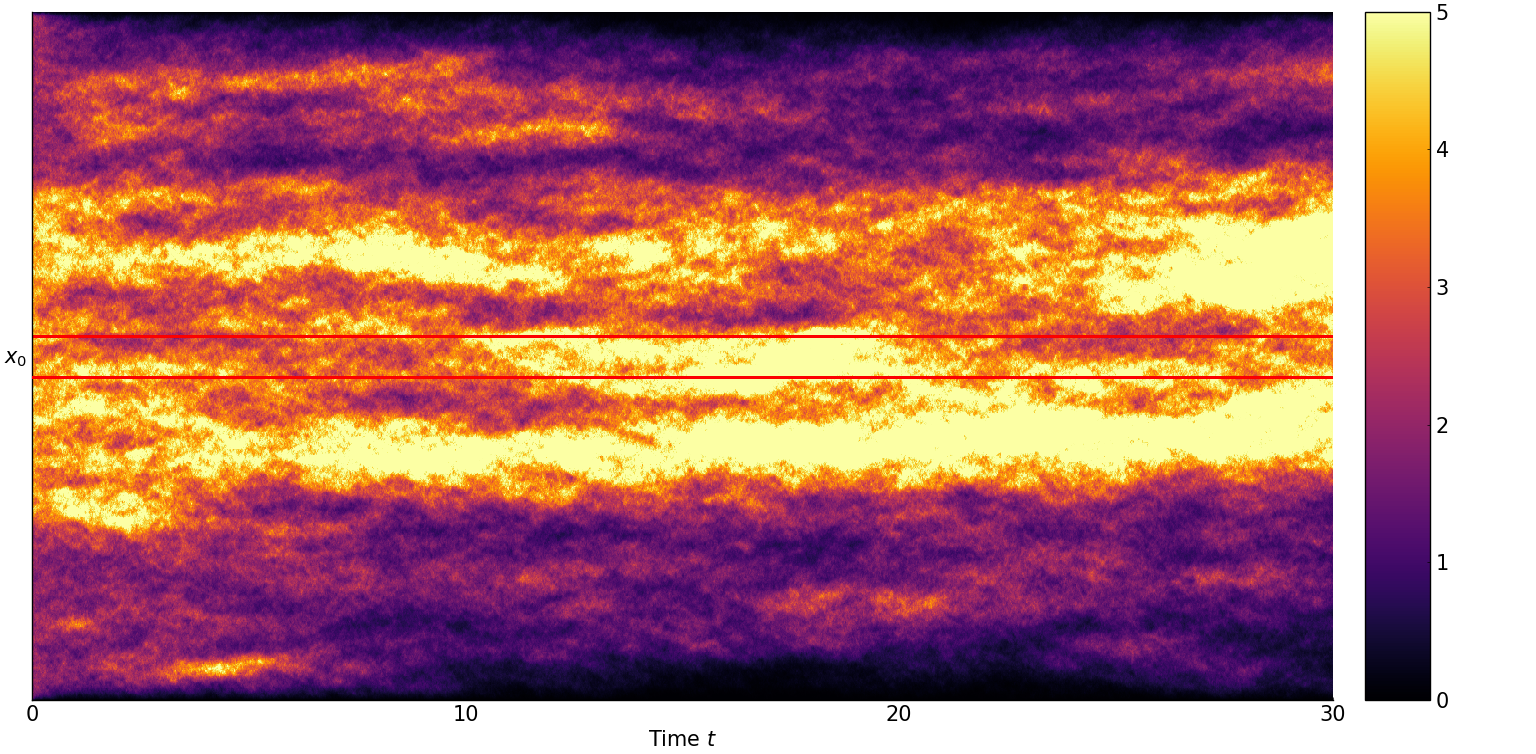}\includegraphics[width=0.5\textwidth]{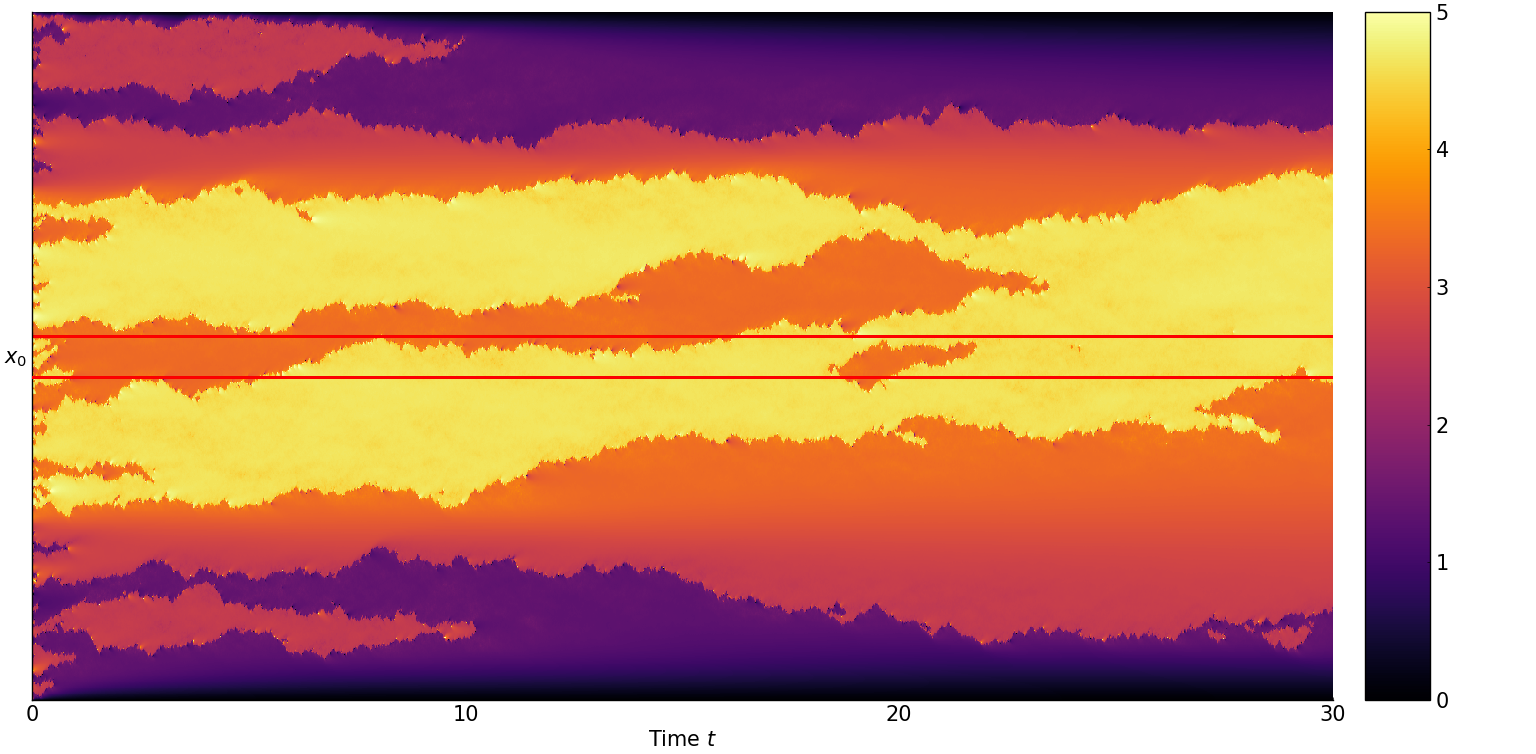}
\caption{Realisation of the stochastic heat equation with  multiplicative noise\\ $\sigma_2 (x) = 0.20 \times |x|^{0.8} + 0.01$ (left) and $\sigma_3 (x) = 10e^{-10 |x-2|} + 10 e^{-10 |x-4|}$ (right).\\ The horizontal lines indicate the support of the kernel $K_{\delta,x_0}$.}
\label{fig:1+2}
\end{figure}

As in \cite{ABJR} we employ the smooth compactly supported kernel
$$
K(x) = \frac{\tilde{K}(x) }{ \| \tilde{K} \|_{L^2(\mathbb R)}}\quad\text{ with }\quad\tilde{K}(x) = \exp \Big( - \frac{10}{1-x^2} \Big) \mathbf{1}_{[-1,1]}(x), \quad x \in \mathbb R,
$$
and we localise around the central point $x_0 = L/2$ with $\delta = 0.03 \times L$. Based on these local measurements, the estimators $\hat{\vartheta}_{\delta}$ (ANE), $\tilde{\vartheta}_{\delta}$ (MNE) and $\vartheta_{\delta}^{\star}$ (SMNE) are computed.

The term $Y(t):=\| \sigma (X(t)) K_{\delta, x_0} \|^2$ is accessed by the following procedure. In view of \eqref{eq:QV}, $Y(t)$ presents the spot squared volatility of $X_{\delta,x_0}$ at time $t$, which we estimate by
$$
\hat{Y}(t_n) = \frac{1}{n \wedge D} \sum_{j=(n-D+1) \vee 1}^{n} \frac{N}{T} (X_{\delta, x_0}(t_j) - X_{\delta, x_0}(t_{j-1}))^2, \quad n = 1, \ldots, N,
$$
i.e., by taking the average disintegrated realised quadratic variation over the past $D = 800$ values ($=0.5$ time units). It is a kernel type estimator of spot squared volatility that follows classical methods, see e.g.  \cite{FJZZ}, Section 2, for a description  and further references. The one-sided estimation kernel is employed so that only historical data are used in the construction and the averaging acts as a smoothing, putting the same weights on the past $D$ values.

Finally, we choose the stabilising value $\varepsilon_{\delta}^2 = \frac{0.001}{\log(10/\delta)}$ such  that it satisfies condition \eqref{eq:epsilon} and lies within the range of typical values of $\| \sigma (X(\cdot)) K_{\delta, x_0} \|^2$. The possible issue could be that if the term $\varepsilon_{\delta}^2$ is much smaller than $\| \sigma (X(\cdot)) K_{\delta, x_0} \|^2$, the SMNE would practically become the MNE. On the other hand, if the term $\varepsilon_{\delta}^2$  dominated $\| \sigma (X(\cdot)) K_{\delta, x_0} \|^2$, then the SMNE would practically coincide with the ANE. So, in practice we recommend to estimate the spot volatility first and then to adjust $\varepsilon_{\delta}^2$ accordingly.

\begin{figure}[t]
	\includegraphics[width=\textwidth]{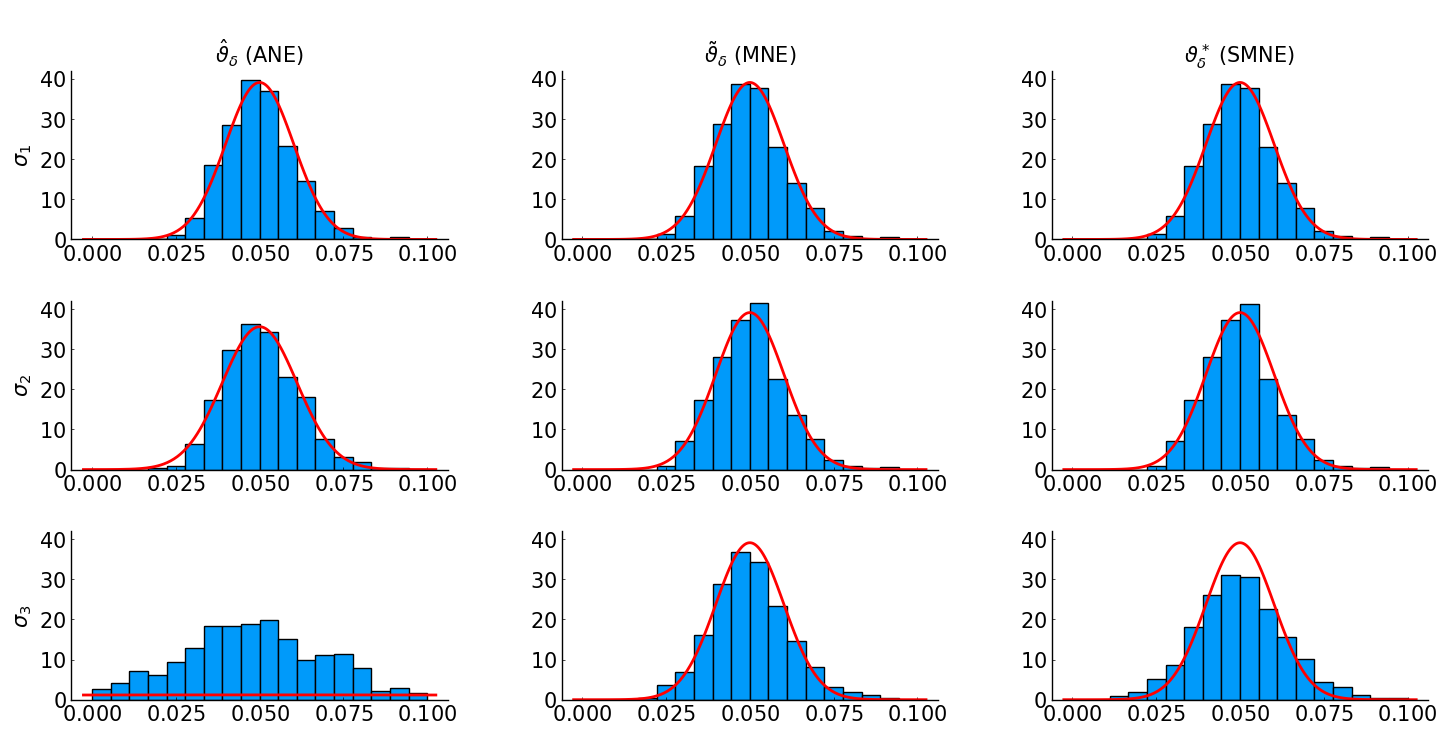}
	\caption{Histograms of the estimators with red lines depicting  the asymptotic densities. \\
From left to right: ANE $\hat{\vartheta}_{\delta}$; MNE $\tilde{\vartheta}_{\delta}$; SMNE $\vartheta_{\delta}^{\star}$. From top to bottom: $\sigma_1$; $\sigma_2$; $\sigma_3$.}
\label{fig:hist}
\end{figure}

\begin{table}[b]
\centering
\begin{tabular}{c|ccc}
\toprule
Mean (SD) & ANE $\hat{\vartheta}_{\delta}$ & MNE $\tilde{\vartheta}_{\delta}$  & SMNE $\vartheta_{\delta}^{\star}$  \\
\hline
$\sigma_1$ & 0.050183 (0.0104) & 0.050057 (0.0104) & 0.050059 (0.0104) \\
$\sigma_2$ & 0.050632 (0.0111) & 0.050163 (0.0104) & 0.050163 (0.0104) \\
$\sigma_3$ & 0.052085 (0.0522) & 0.050115 (0.0115) & 0.050232 (0.0135) \\
\bottomrule
\end{tabular}
\caption{Monte Carlo mean and standard deviation of estimators for different $\sigma(\cdot)$.}
\label{tbl:1}
\end{table}

Figure \ref{fig:hist} displays simulation results for the estimators of the parameter $\vartheta$ obtained after $1 \, 000$ Monte Carlo runs for each of the functions $\sigma_1$, $\sigma_2$ and $\sigma_3$. The red lines in the histograms indicate the  asymptotic distribution, obtained as a mixture of $1 \, 000$ Gaussian densities that (individually, for each run) follow the theoretical results established in Theorems \ref{thm:hatvartheta}, \ref{thm:tildevartheta} and \ref{thm:starvartheta}. Monte Carlo mean and standard deviation for every case are stated in Table \ref{tbl:1}.

In the additive noise case $\sigma_1$ all three estimators perform similarly well. In this case we have  equalities in \eqref{eq:asympvariances} and the resulting asymptotic distributions coincide.
In the ``H\"older'' multiplicative noise case $\sigma_2$, the estimator ANE performs slightly worse than the two alternatives. Since $\sigma_2(\cdot) \geq  0.01 > 0$, we have $T^{\star} = T$ and the estimators MNE and SMNE deliver similar results.

For $\sigma_3$ the histogram of the ANE in Figure \ref{fig:hist} (bottom, left) is much more spread out, but has not yet entered the asymptotic regime with a very flat asymptotic density. There are, however, quite a few outliers (12.6 \%) outside the interval $[0,0.1]$, which are not shown and which are caught pretty well by the tails of the asymptotic density. Note also that the corresponding empirical standard deviation in Table \ref{tbl:1} is very high with about  half the length of the interval $[0,0.1]$. The estimators MNE and SMNE give a significant improvement here with an error distribution that is almost unchanged with respect to the cases $\sigma_1$ and $\sigma_2$. It is worth noting that the assumption $\underline{\sigma} > 0$ from Theorem~\ref{thm:tildevartheta} for the MNE is violated by $\sigma_3$ and we also had $T^{\star} < T$, but with a minor difference only. In the discrete numerical setting we use the threshold $10^{-6}$  to determine whether $\sigma(X(t, x_0))$ is zero or not.

\begin{figure}
\centering
	\includegraphics[width=0.5\textwidth]{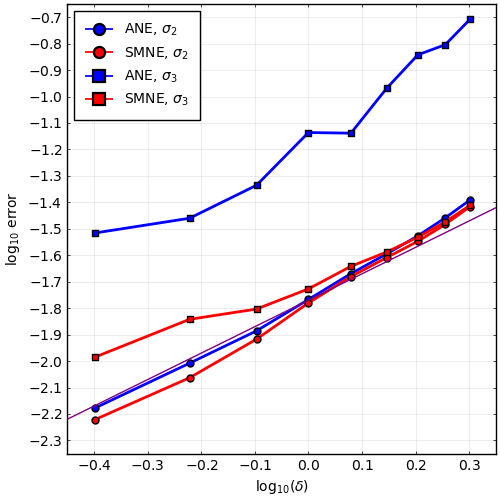}
	\caption{$\log_{10}$-$\log_{10}$ plot of RMSE for the estimators ANE $\hat{\vartheta}_{\delta}$ and SMNE $\vartheta_{\delta}^{\star}$ under multiplicative noise $\sigma_2(\cdot)$ and $\sigma_3(\cdot)$. The purple line with slope $1$ is added as reference.}
\label{fig:log}
\end{figure}

Simulation results for varying $\delta$ confirm the convergence rate $\delta$ as $\delta \rightarrow 0$. Figure \ref{fig:log} shows a $\log_{10}$-$\log_{10}$ plot of root mean squared estimation errors for the estimators ANE $\hat{\vartheta}_{\delta}$ and SMNE $\vartheta_{\delta}^{\star}$ obtained after $100$ Monte Carlo simulations for each $\delta$ based on the multiplicative noise $\sigma_2$ and $\sigma_3$. The estimator MNE $\tilde{\vartheta}_{\delta}$ is omitted here, because under the noise $\sigma_2$ the differences between MNE and SMNE are minimal and the assumption $\underline{\sigma} > 0$ for the MNE is again violated by $\sigma_3$. The estimation errors are significantly smaller in the ``H\"older'' multiplicative noise case $\sigma_2$ and the SMNE provides a very substantial improvement in the $\sigma_3$ case. The errors are very well aligned with the asymptotic standard error as predicted by Theorems \ref{thm:hatvartheta} and \ref{thm:starvartheta}. The  reference line with slope $1$ is added to compare with the theoretical convergence rate $\delta$. Note that the spatial discretisation for reliable simulations must always be much finer than $\delta$. In our setup with $\delta = 0.6$ and $L/M = 0.025$, the localised kernel $K_{\delta, x_0}$ was evaluated discretely on $48$ grid points.

Further unreported simulations show that the performance of the estimators is not influenced by the location of the central observation point $x_0$, unless $x_0$ is located very close to the boundary. In fact, if  several local measurements (localised around points $\{ x_0^j: j = 1, \ldots J \}$) are available, it is possible to combine local estimators, see \cite{ABJR}, where such an approach is explained and used.

In conclusion, the two newly proposed estimators MNE and SMNE performed as well as the ANE or even better than the ANE and their asymptotic distribution matches the results obtained in Section \ref{sec:mainresults}. The ANE provides good estimation accuracy also under multiplicative noise, but its accuracy suffers under strongly varying $\sigma(\cdot)$.

\section{Proofs}\label{sec:proofs}

First we shall establish all results under the additional condition
\begin{equation}\label{Eqbarsigma}
\overline{\sigma} := \sup_{x \in \mathbb R} \sigma (x) < \infty.
\end{equation}
Using the continuity of the solution $X(t,x)$, we shall then get rid of this assumption in the last step of the proofs of central limit theorems in Section \ref{app:stableconv}.

\subsection{Fundamental asymptotics}

We need some properties of the rescaled operators and semigroups. Let $(S_{\vartheta, \delta, x_0}(t), t \geq 0)$ be the strongly continuous semigroup generated by $\vartheta \Delta$ with Dirichlet boundary conditions on $L^2(\Lambda_{\delta, x_0})$ and note that both semigroups $(S_{\vartheta}(t), t \geq 0)$ and $(S_{\vartheta, \delta, x_0}(t), t \geq 0)$, are self-adjoint. We cite Lemma 3.1 from \cite{AR}:

\begin{lemma} \label{scaling properties}
For $\delta \in(0,1)$ the following holds:
\begin{itemize}
\item[(i)] If $z \in H_0^1(\Lambda_{\delta, x_0}) \cap H^2(\Lambda_{\delta, x_0})$, then $\Delta z_{\delta, x_0} = \delta^{-2} \left( \Delta z \right)_{\delta, x_0}$.
\item[(ii)] If $z \in L^2(\Lambda_{\delta, x_0})$, then $S_{\vartheta}(t) z_{\delta, x_0} = \left( S_{\vartheta, \delta, x_0}(\delta^{-2} t) z \right)_{\delta, x_0}$, $t \geq 0$.
\end{itemize}
\end{lemma}

The deterministic flow of the initial condition will become negligible due to the next result.

\begin{lemma} \label{lemma:ini}
For an initial condition $X_0\in C(\bar\Lambda)$ we have
$$
\int_0^T \left\langle S_{\vartheta} (t) X_0, \Delta K_{\delta, x_0} \right\rangle^2 \, dt =O(\delta^{-11/6})=o(\delta^{-2})\text{  as $\delta \rightarrow 0$.}
$$
\end{lemma}

\begin{proof}
Lemma A.7(ii) in \cite{AR} shows for $X_0\in L^p(\Lambda)$, $p\ge 2$, that
$$
\delta^4 \int_0^T \left\langle S_{\vartheta} (t) X_0, \Delta K_{\delta, x_0} \right\rangle^2 \, dt \lesssim \| X_0 \|_{L^p(\Lambda)}^2 \delta^{(7p-2)/(3p)}.
$$
Because of $X_0\in C(\bar\Lambda)\subset L^\infty(\Lambda)$ we may choose $p=4$ and obtain the result.
\end{proof}

\begin{lemma} \label{lemma: limit sigma K}
Grant Assumptions \ref{ass:sigma} with \eqref{Eqbarsigma} and \ref{ass:X}. For any $t \in [0, T]$,
$$
\| \sigma(X(t)) K_{\delta, x_0} \|^2 \rightarrow \sigma^2(X(t, x_0)) \| K \|_{L^2(\mathbb R)}^2
$$
holds as $\delta \rightarrow 0$. The limit holds true for any $\omega \in \Omega$, i.e., surely.
\end{lemma}
\begin{proof}
For any $t \in [0, T]$, we have by the continuity of $\sigma$ and $X$, provided by Assumptions \ref{ass:sigma} and \ref{ass:X},
\begin{align*}
\| \sigma(X(t)) K_{\delta, x_0} \|^2 &= \int_{\Lambda} \sigma^2(X(t, x)) K_{\delta, x_0}^2(x) \, dx  \\
&= \int_{\mathbb R} \sigma^2(X(t, \delta y + x_0)) K^2(y) \mathbf{1}_{\Lambda_{\delta, x_0}}(y) \, dy \\
&\underset{\delta \rightarrow 0}{\longrightarrow} \int_{\mathbb R} \sigma^2(X(t, x_0)) K^2(y) \, dy.
\end{align*}
Here the dominated convergence theorem is applied with integrable majorant $\overline{\sigma}^2 K^2(\cdot)$.
\end{proof}

We will  need an inequality for the rescaled semigroup $S_{\vartheta, \delta, x_0}$. It encapsulates essentially the hypercontractivity of the heat semigroup.

\begin{lemma} \label{lemma:JJinequality}
For $\alpha \in [0,2]$, $\delta\in(0,1)$ and $0 \leq v \leq \delta^{-2}T$ we have
$$
\| |x|^{\alpha} S_{\vartheta, \delta, x_0} (v) \Delta K \|_{L^2(\Lambda_{\delta, x_0})} \lesssim (1+v)^{3(\alpha-1)/4}.
$$
\end{lemma}
\begin{proof}
For $\alpha = 0$ we apply Lemma A.6(ii) from \cite{AR} with $w = \Delta K$, $\alpha' = 1$ and $d=1$. The constant involved only depends on $T$ and $K$, which are fixed here.

For $\alpha = 2$ we combine Proposition 3.5(i) in \cite{AR} and the second inequality from Lemma A.2(iii) in \cite{AR} such that
\begin{align*}
\| |x|^2 S_{\vartheta, \delta, x_0} (v) \Delta K \|_{L^2(\Lambda_{\delta, x_0})}
&\lesssim (1 \vee v) (1 \wedge v^{-1/4}) \| (1 + |x| + |x|^2) \Delta K \|_{L^1 \cap L^2(\mathbb R)}\\
&\lesssim (1+v)^{3/4}
\end{align*}
with $\norm{z}_{L^1 \cap L^2(\R)}:=\norm{z}_{L^1(\R)}+\norm{z}_{L^2(\R)}$,
using that $K\in H^2(\R)$ has compact support and is fixed throughout.

For $\alpha \in (0,2)$ we use the H\"older inequality with weight function $w(x) =  (S_{\vartheta, \delta, x_0} (v) \Delta K)(x)^2$ and $p = 2/\alpha$ to obtain
\begin{align*}
&\| |x|^{\alpha} S_{\vartheta, \delta, x_0} (v) \Delta K \|_{L^2(\Lambda_{\delta, x_0})}^2
\\
&\leq \Big( \int_{\Lambda_{\delta, x_0}} |x|^4  (S_{\vartheta, \delta, x_0} (v) \Delta K)(x)^2 \, dx \Big)^{\alpha/2} \Big( \int_{\Lambda_{\delta, x_0}}  (S_{\vartheta, \delta, x_0} (v) \Delta K)(x) ^2 \, dx \Big)^{1 - \alpha/2} \\
&= \| |x|^2 S_{\vartheta, \delta, x_0} (v) \Delta K \|_{L^2(\Lambda_{\delta, x_0})}^{\alpha} \| S_{\vartheta, \delta, x_0} (v) \Delta K \|_{L^2(\Lambda_{\delta, x_0})}^{2 - \alpha} \\
&\lesssim (1+v)^{3\alpha/4}(1+v)^{-(6-3\alpha)/4}= (1+v)^{3(\alpha-1)/2} ,
\end{align*}
where we used the first two parts of the proof.
\end{proof}

In the sequel, we need a uniform bound on the second and fourth centered moment of $X_{\delta, x_0}^{\Delta}(t)$.

\begin{lemma} \label{lemma:fourth moment}
Grant Assumptions \ref{ass:sigma} with \eqref{Eqbarsigma} and \ref{ass:X}. For $\delta \rightarrow 0$ we have
$$
\sup_{0 \leq t \leq T} \mathbb E \left( X_{\delta, x_0}^{\Delta} (t) -\E X_{\delta, x_0}^{\Delta} (t) \right)^4 = O(\delta^{-4}),
$$
in particular we have $\Var( X_{\delta, x_0}^{\Delta} (t)) = O(\delta^{-2})$ uniformly over $t\in[0,T]$.
\end{lemma}

\begin{proof}
By \eqref{X-delta-delta} and \eqref{mild solution} we have
$$
X_{\delta, x_0}^{\Delta} (t)- \E X_{\delta, x_0}^{\Delta} (t) = \int_0^t \left\langle \sigma(X(s)) S_{\vartheta} (t-s) \Delta K_{\delta, x_0}, dW(s) \right\rangle.
$$

The Burkholder-Davis-Gundy inequality yields with a constant $C_4\ge 1$
\begin{equation} \label{BDG}
\mathbb E \Big( \int_0^t \left\langle g(t_0, s), dW(s) \right\rangle \Big)^4 \leq C_4 \mathbb E \Big( \int_0^t \| g(t_0, s) \|^2 \, ds \Big)^{2}
\end{equation}
for  $g(t_0, s) = \sigma(X(s)) S_{\vartheta} (t_0-s) \Delta K_{\delta, x_0}$. With $t_0 = t$ we obtain
\begin{align}
\mathbb E \left( X_{\delta, x_0}^{\Delta} (t) -\E X_{\delta, x_0}^{\Delta} (t) \right)^4 &\leq C_4 \mathbb E \Big( \int_0^t \left\| \sigma(X(s)) S_{\vartheta} (t-s) \Delta K_{\delta, x_0} \right\|^2 \, ds \Big)^2 \notag \\
&\leq C_4 \overline{\sigma}^4 \Big( \int_0^t \left\| S_{\vartheta} (t-s) \Delta K_{\delta, x_0} \right\|^2 \, ds \Big)^2 \notag \\
&\leq C_4 \overline{\sigma}^4 \Big(\theta^{-1} \int_0^\infty \scapro{S_{\vartheta} (2v) \theta\Delta K_{\delta, x_0}} {\Delta K_{\delta, x_0}} \, dv \Big)^2 \notag \\
&= C_4 \overline{\sigma}^4 (2\theta)^{-2} \scapro{- K_{\delta, x_0}} {\Delta K_{\delta, x_0}}^2 \notag \\
&= C_4 \overline{\sigma}^4 (2\theta)^{-2}\delta^{-4} \norm{K'}_{L^2(\R)}^4. \label{ineq for BDG}
\end{align}
Note that inequality \eqref{ineq for BDG} holds uniformly over $t\in[0,T]$ and gives the result. The second part follows via Jensen's inequality.
\end{proof}

\subsection{Approximation of quadratic variations and related terms}

We are ready to study the asymptotic behaviour of the terms $\mathcal J_{\delta}$, $\mathcal I_{\delta}$, $\tilde{\mathcal I}_{\delta}$, $\mathcal J_{\delta}^{\star}$ and $\mathcal I_{\delta}^{\star}$. We will analyse them simultaneously, using their common generalisation $\mathcal L_{\delta}$.

First, let for any $\delta \in (0,1)$, $f_{\delta}: L^2(\Lambda) \rightarrow \mathbb R$ be a continuous (and possibly non-linear) functional of the state, satisfying one of the following:
\begin{itemize}
\item[(F1)] There exists $C > 0$ such that for any $z, y \in L^2(\Lambda)$, $\delta\in(0,1)$
\begin{align*}
 |f_{\delta}(z)| &\leq C, \\
|f_{\delta}(z) - f_{\delta}(y)| &\leq C \| (\sigma(z) - \sigma(y)) K_{\delta, x_0} \|,
\end{align*}
\item[(F2)] There exists $C > 0$ such that for any $z, y \in L^2(\Lambda)$, $\delta\in(0,1)$
\begin{align*}
|f_{\delta}(z)| &\leq C \varepsilon_{\delta}^{-4}, \\
|f_{\delta}(z) - f_{\delta}(y)| &\leq C \varepsilon_{\delta}^{-8} \| (\sigma(z) - \sigma(y)) K_{\delta, x_0} \|,
\end{align*}
where $(\varepsilon_{\delta})$ is  satisfying \eqref{eq:epsilon}.
\end{itemize}

\begin{lemma} \label{lemma:f} Grant Assumption \ref{ass:sigma} with \eqref{Eqbarsigma}.
\begin{itemize}
\item[(i)] Functionals $f_{\delta}(\cdot) \in \{ 1, \| \sigma(\cdot) K_{\delta, x_0} \|^2, \| \sigma(\cdot) K_{\delta, x_0} \|^{-2} \}$ satisfy condition (F1), provided $\underline\sigma>0$ in the last case.
\item[(ii)] Functionals $f_{\delta}(\cdot) \in \{ \frac{1}{\| \sigma(\cdot) K_{\delta, x_0} \|^2 + \varepsilon_{\delta}^2}, \frac{\| \sigma(\cdot) K_{\delta, x_0} \|^2}{(\| \sigma(\cdot) K_{\delta, x_0} \|^2 + \varepsilon_{\delta}^2)^2}  \}$ satisfy condition (F2).
\end{itemize}
\end{lemma}
\begin{proof}
(i). For $f_{\delta}(\cdot) \equiv 1$ the statement is trivial.

For $f_{\delta}(\cdot) = \| \sigma(\cdot) K_{\delta, x_0} \|^2$ we have a uniform upper bound $\overline{\sigma}^2 \| K \|_{L^2(\mathbb R)}^2$. Moreover, for any $z, y \in L^2(\Lambda)$ and $\delta \in(0,1)$ we have
\begin{align*}
\abs{f_{\delta}(z) - f_{\delta}(y)} &= \abs{ \| \sigma(z) K_{\delta, x_0} \|^2 - \| \sigma(y) K_{\delta, x_0} \|^2} \\
&\lesssim  \abs{\| \sigma(z) K_{\delta, x_0} \| - \| \sigma(y) K_{\delta, x_0} \|} \\
&\leq \| (\sigma(z) - \sigma(y)) K_{\delta, x_0} \|,
\end{align*}
where we used  $|A^2 - B^2| = |A-B| |A+B|$ together with the upper bound $2 \overline{\sigma} \| K \|_{L^2(\mathbb R)}$ in the first inequality and then we followed up with the reverse triangle inequality.

For $f_{\delta}(\cdot) = \| \sigma(\cdot) K_{\delta, x_0} \|^{-2}$, there is a uniform upper bound $\underline{\sigma}^{-2} \| K \|_{L^2(\mathbb R)}^{-2}$. Moreover, for any $z, y \in L^2(\Lambda)$ and $\delta \in(0,1)$ we have
\begin{align*}
\babs{\frac{1}{\| \sigma(z) K_{\delta, x_0} \|^2} - \frac{1}{\| \sigma(y) K_{\delta, x_0} \|^2} } &= \babs{ \frac{\| \sigma(y) K_{\delta, x_0} \|^2 - \| \sigma(z) K_{\delta, x_0} \|^2}{\| \sigma(z) K_{\delta, x_0} \|^2 \| \sigma(y) K_{\delta, x_0} \|^2} } \\
&\lesssim \abs{ \| \sigma(y) K_{\delta, x_0} \|^2 - \| \sigma(z) K_{\delta, x_0} \|^2 }.
\end{align*}
Therefore the proof is finished as in the previous case.

(ii). For $f_{\delta}(\cdot) = \frac{1}{\| \sigma(\cdot) K_{\delta, x_0} \|^2 + \varepsilon_{\delta}^2}$, its upper bound $\varepsilon_{\delta}^{-2}$  is smaller than $\varepsilon_{\delta}^{-4}$. To prove the second part, we use $|A^2 - B^2| = |A + B| |A - B|$, the upper bound $\overline{\sigma} \| K \|_{L^2(\mathbb R)}$ to $\| \sigma(\cdot) K_{\delta, x_0} \|$ and the reverse triangle inequality to obtain
\begin{align*}
\abs{f_\delta(y)-f_\delta(z)}
&= \babs{\frac{\| \sigma(y) K_{\delta, x_0} \|^2 - \| \sigma(z) K_{\delta, x_0} \|^2}{(\| \sigma(z) K_{\delta, x_0} \|^2 + \varepsilon_{\delta}^2) (\| \sigma(y) K_{\delta, x_0} \|^2 + \varepsilon_{\delta}^2)} }  \\
&\leq \varepsilon_{\delta}^{-4} \babs{ \| \sigma(y) K_{\delta, x_0} \|^2 - \| \sigma(z) K_{\delta, x_0} \|^2} \\
&\lesssim \varepsilon_{\delta}^{-4} \abs{ \| \sigma(y) K_{\delta, x_0} \| - \| \sigma(z) K_{\delta, x_0} \|} \\
&\le \varepsilon_{\delta}^{-4} \| \left( \sigma(y) - \sigma(z) \right) K_{\delta, x_0} \|.
\end{align*}

For $f_{\delta}(\cdot) = \frac{\| \sigma(\cdot) K_{\delta, x_0} \|^2}{(\| \sigma(\cdot) K_{\delta, x_0} \|^2 + \varepsilon_{\delta}^2)^2}$, there is a uniform upper bound $\overline{\sigma}^2 \| K \|_{L^2(\mathbb R)}^2 \varepsilon_{\delta}^{-4}$. Similarly to the previous case, algebraic calculations yield
$$
\abs{f_\delta(y)-f_\delta(z)}\lesssim \varepsilon_{\delta}^{-8} \babs{ \| \sigma(y) K_{\delta, x_0} \|^2 - \| \sigma(z) K_{\delta, x_0} \|^2},
$$
which finishes the proof.
\end{proof}

Introduce
\begin{align}
\mathcal L_{\delta} &= \int_0^T f_{\delta}(X(t)) \left( X_{\delta, x_0}^{\Delta} (t)-\E X_{\delta, x_0}^{\Delta} (t) \right)^2 \, dt \notag \\
&= \int_0^T f_{\delta}(X(t)) \Big( \int_0^t \left\langle S_{\vartheta} (t-s) \Delta K_{\delta, x_0}, \sigma(X(s)) dW(s) \right\rangle \Big)^2 \, dt.\label{eq:Ldelta}
\end{align}
In the following proposition we present different expressions that are equal to $\mathcal L_{\delta}$ up to terms that are of lower order than $\delta^{-2}$. This is the major ingredient for the proofs of the main results, noting that the techniques developed in \cite{AR} cannot be used here due to the multiplicative noise structure.

\begin{proposition} \label{prop:K}
Grant Assumptions \ref{ass:sigma} with \eqref{Eqbarsigma}, \ref{ass:X} and let $f_{\delta}$ satisfy condition (F1) or grant Assumptions \ref{ass:sigma'} with \eqref{Eqbarsigma}, \ref{ass:X'} with \eqref{eq:epsilon} and let $f_{\delta}$ satisfy condition (F2). Then $\mathcal L_{\delta}$ from \eqref{eq:Ldelta} equals up to additive terms of order $o_{\mathbb P}(\delta^{-2})$ for $\delta \rightarrow 0$:
\begin{itemize}
\item[(i)] $\mathcal L_{\delta}^{(i)} = \int_{\delta}^T f_{\delta}(X(t)) ( X_{\delta, x_0}^{\Delta} (t) -\E X_{\delta, x_0}^{\Delta} (t) )^2 \, dt,$
\item[(ii)] $\mathcal L_{\delta}^{(ii)} = \int_{\delta}^T f_{\delta}(X(t - \delta)) ( X_{\delta, x_0}^{\Delta} (t) -\E X_{\delta, x_0}^{\Delta} (t) )^2 \, dt,$
\item[(iii)] $ \mathcal L_{\delta}^{(iii)} = \int_{\delta}^T f_{\delta}(X(t - \delta)) ( \int_0^t \sigma(X(s, x_0)) \langle S_{\vartheta} (t-s) \Delta K_{\delta, x_0}, dW(s) \rangle )^2 \, dt$,
\item[(iv)] $\mathcal L_{\delta}^{(iv)} = \int_{\delta}^T f_{\delta}(X(t - \delta)) ( \int_{t-\delta}^t \sigma(X(s, x_0)) \langle S_{\vartheta} (t-s) \Delta K_{\delta, x_0}, dW(s) \rangle )^2 \, dt$,
\item[(v)] $\mathcal L_{\delta}^{(v)} = \int_{\delta}^T f_{\delta}(X(t - \delta)) \sigma^2(X(t-\delta, x_0))  ( \int_{t-\delta}^t \langle S_{\vartheta} (t-s) \Delta K_{\delta, x_0}, dW(s) \rangle )^2 \, dt$,
\item[(vi)] $\mathcal L_{\delta}^{(vi)} = \int_{\delta}^T f_{\delta}(X(t - \delta)) \sigma^2(X(t-\delta, x_0)) \E ( \int_{t-\delta}^t \langle S_{\vartheta} (t-s) \Delta K_{\delta, x_0}, dW(s) \rangle )^2 \, dt$.
\item[(vii)] $\mathcal L_{\delta}^{(vii)} = (2\theta)^{-1}\norm{K'}_{L^2(\mathbb R)}^2\delta^{-2}\int_{0}^{T} f_{\delta}(X(t)) \sigma^2(X(t, x_0)) \, dt$.
\end{itemize}
\end{proposition}

\begin{remark}
The overall idea is to achieve the representation $\mathcal L_{\delta}^{(vii)}$ via slight con\-se\-cu\-tive alterations. In point (i) we shorten the outer integral to the interval $[\delta, T]$, in (ii), we present a slight time shift of the solution in the functional, i.e., $f_{\delta}(X(t - \delta))$. In  point (iii) the function $\sigma(X(s, \cdot))$ is fixed in the space-point $x_0$, in (iv) the stochastic integral is shortened, in (v) the function $\sigma(X(s, x_0))$ is fixed at the time-point $t - \delta$. In (vi) the expectation of the squared stochastic integral is implemented via conditional independence. Finally, in (vii) the expectation is approximated and the integral extended again.
\end{remark}

\begin{proof}
We present the proof with Assumptions \ref{ass:sigma'} with \eqref{Eqbarsigma}, \ref{ass:X'} with \eqref{eq:epsilon} and $f_{\delta}$ satisfying condition (F2). The other case is analogous and much simpler, mostly because it does not use the function $\varepsilon_{\delta}$ at all. The proof of (ii) is shown for both cases.
The order $O( \varepsilon_{\delta}^{-4})$ for $|f_{\delta}|$ is used frequently as the first step of the proof.

(i). Compute
\begin{align}
&\mathbb E \left| \mathcal L_{\delta} - \mathcal L_{\delta}^{(i)} \right| = \mathbb E \Big| \int_0^{\delta} f_{\delta}(X(t)) \left( X_{\delta, x_0}^{\Delta} (t)-\E X_{\delta, x_0}^{\Delta} (t) \right)^2 \, dt \Big| \label{Ki} \\
&\lesssim \varepsilon_{\delta}^{-4} \int_0^{\delta} \Var\left( X_{\delta, x_0}^{\Delta} (t) \right) \, dt \lesssim \eps_\delta^{-4}\delta\delta^{-2},\notag
\end{align}
using Lemma \ref{lemma:fourth moment}. Therefore, the remainder term is of order $o_{\mathbb P}(\delta^{-2})$ due to $\varepsilon_{\delta}^{-4} \delta\to 0$  by \eqref{eq:epsilon}.

(ii). By the Cauchy-Schwarz inequality, we have
\begin{align}
&\left| \mathcal L_{\delta}^{(i)} - \mathcal L_{\delta}^{(ii)} \right| \notag \\
& \leq  \int_{\delta}^T \left| f_{\delta}(X(t)) - f_{\delta}(X(t - \delta)) \right| \left( X_{\delta, x_0}^{\Delta} (t)-\E X_{\delta, x_0}^{\Delta} (t) \right)^2 \, dt \notag \\
&\leq \Big(\int_{\delta}^T  \left( f_{\delta}(X(t)) - f_{\delta}(X(t - \delta)) \right)^2 dt \Big)^{1/2} \Big( \int_\delta^T  \left(X_{\delta, x_0}^{\Delta} (t)-\E X_{\delta, x_0}^{\Delta} (t) \right)^4 dt\Big)^{1/2}. \label{Kiii}
\end{align}

Lemma \ref{lemma:fourth moment} gives $ \mathbb E \int_\delta^T( X_{\delta, x_0}^{\Delta} (t)-\E X_{\delta, x_0}^{\Delta} (t) )^4 \, dt = O(\delta^{-4})$ so that the second factor is of order $O_{\PP}(\delta^{-2})$. Hence, it suffices to establish  $\int_{\delta}^T  ( f_{\delta}(X(t)) - f_{\delta}(X(t - \delta)) )^2 dt=o_{\PP}(1)$.

When we consider Assumptions \ref{ass:sigma'}, \ref{ass:X'}, condition \eqref{Eqbarsigma} and $f_{\delta}$ satisfying condition (F2), we obtain
\begin{align*}
 &\E\int_\delta^T \left( f_{\delta}(X(t)) - f_{\delta}(X(t - \delta)) \right)^2dt\\ &\lesssim \varepsilon_{\delta}^{-16}  \int_\delta^T \E \| \left( \sigma(X(t)) - \sigma(X(t - \delta)) \right) K_{\delta, x_0} \|^2 dt
\lesssim \varepsilon_{\delta}^{-16} \delta^{2 \beta_t \beta_{\sigma}} T\| K \|_{L^2(\mathbb R)}^2
\end{align*}
by the H\"older continuity of $\sigma$ and $X(\cdot,x)$ and by $\norm{K_{\delta, x_0}}=\norm{K}_{L^2(\R)}$. This upper bound converges to zero by \eqref{eq:epsilon}.

When we consider Assumptions \ref{ass:sigma}, \ref{ass:X}, condition \eqref{Eqbarsigma} and $f_{\delta}$ satisfying condition (F1), we have
\begin{align}
& \left( f_{\delta}(X(t)) - f_{\delta}(X(t - \delta)) \right)^2 \lesssim  \| \left( \sigma(X(t)) - \sigma(X(t - \delta)) \right) K_{\delta, x_0} \|^2 \notag \\
&= \int_{\Lambda} \left( \sigma(X(t, x)) - \sigma(X(t - \delta, x)) \right)^2 K_{\delta, x_0}^2 (x) \, dx \notag \\
&=  \int_{\mathbb R} \left( \sigma(X(t, \delta y + x_0)) - \sigma(X(t - \delta, \delta y + x_0)) \right)^2 K^2(y) \, dy. \label{Kiiib}
\end{align}

The integrand in \eqref{Kiiib} converges to zero by the continuity of $X$ and $\sigma$ (i.e., Assumptions \ref{ass:X} and \ref{ass:sigma}). The integrable majorant $4 \overline{\sigma}^2 K^2(\cdot)$ serves for the  $dy$--integral in \eqref{Kiiib} as well as for the $dt$--integral in \eqref{Kiii}.  The proof that the second factor converges almost surely to zero is accomplished by the dominated convergence theorem.

(iii). By the upper bound on $|f_{\delta}|$, we have
\begin{align}
&\mathbb E \left| \mathcal L_{\delta}^{(ii)} - \mathcal L_{\delta}^{(iii)} \right| \lesssim  \varepsilon_{\delta}^{-4} \mathbb E \int_{\delta}^T \Big| \Big( \int_0^t \left\langle S_{\vartheta} (t-s) \Delta K_{\delta, x_0}, \sigma(X(s)) dW(s) \right\rangle \Big)^2 -  \notag \\
& - \Big( \int_0^t \left\langle S_{\vartheta} (t-s) \Delta K_{\delta, x_0}, \sigma(X(s, x_0)) dW(s) \right\rangle \Big)^2 \Big| \, dt. \label{Kiv}
\end{align}
Denoting
\begin{equation} \label{eq:Fpm}
F_{\pm} (t) := \Big| \int_0^t \left\langle S_{\vartheta}(t-s) \Delta K_{\delta, x_0}, \left( \sigma(X(s)) \pm \sigma(X(s, x_0)) \right) \, dW(s) \right\rangle \Big|,
\end{equation}
using $|A^2 - B^2| = |A+B| |A-B|$ and  the Cauchy-Schwarz inequality, we may follow up \eqref{Kiv} with
\begin{align}
\mathbb E \left| \mathcal L_{\delta}^{(ii)} - \mathcal L_{\delta}^{(iii)} \right| &\lesssim \varepsilon_{\delta}^{-4} \mathbb E \int_{\delta}^T F_+(t) F_-(t) \, dt \notag \\
&\leq \varepsilon_{\delta}^{-4} \Big( \mathbb E \int_{\delta}^T F_+^2(t) \, dt \Big)^{1/2} \Big( \mathbb E \int_{\delta}^T F_-^2(t) \, dt \Big)^{1/2}. \label{Kiv 2}
\end{align}

We start with the first factor after $\eps_\delta^{-4}$. It\^o's isometry yields
\begin{align*}
\mathbb E \int_{\delta}^T F_+^2(t) \, dt &= \mathbb E \int_{\delta}^T \int_0^t \left\| \left( \sigma(X(s)) + \sigma(X(s, x_0)) \right) S_{\vartheta}(t-s) \Delta K_{\delta, x_0} \right\|^2 \, ds \, dt \\
&\lesssim \int_{\delta}^T \int_0^t \left\| S_{\vartheta}(t-s) \Delta K_{\delta, x_0} \right\|^2 \, ds \, dt,
\end{align*}
and this is of order $O(\delta^{-2})$, compare  \eqref{ineq for BDG}. The second factor is given by
\begin{align}
&\mathbb E \int_{\delta}^T F_-^2(t) \, dt \label{eq:Fm} \\
&= \mathbb E \int_{\delta}^T \int_0^t \left\| \left( \sigma(X(s)) - \sigma(X(s, x_0)) \right) S_{\vartheta}(t-s) \Delta K_{\delta, x_0} \right\|^2 \, ds \, dt \notag \\
&= \delta^{-2} \int_{\delta}^T \int_0^{\delta^{-2} t} \mathbb E \left\| \left( \sigma(X(t - \delta^2 v)) - \sigma(X(t - \delta^2 v, x_0)) \right) \left( S_{\vartheta, \delta, x_0}(v) \Delta K \right)_{\delta, x_0} \right\|^2 \, dv \, dt \notag \\
&= \delta^{-2} \int_{\delta}^T \int_0^{\delta^{-2} t} \mathbb E \left\| \left( \sigma(X(t - \delta^2 v, \delta \cdot + x_0)) - \sigma(X(t - \delta^2 v, x_0)) \right) S_{\vartheta, \delta, x_0}(v) \Delta K \right\|_{L^2(\Lambda_{\delta, x_0})}^2 \, dv \, dt. \notag
\end{align}
Using Assumptions \ref{ass:sigma'} and \ref{ass:X'}, we have
\begin{align*}
& \E\left( \sigma(X(t - \delta^2 v, \delta x + x_0)) - \sigma(X(t - \delta^2 v, x_0)) \right)^2  \lesssim \delta^{2 \beta_x \beta_{\sigma}} |x|^{2 \beta_x \beta_{\sigma}}
\end{align*}
and consequently for the integrand in \eqref{eq:Fm}
\begin{align}
&\mathbb E \left\| \left( \sigma(X(t - \delta^2 v, \delta \cdot + x_0)) - \sigma(X(t - \delta^2 v, x_0)) \right) S_{\vartheta, \delta, x_0}(v) \Delta K \right\|_{L^2(\Lambda_{\delta, x_0})}^2 \notag \\
&\lesssim \delta^{2 \beta_x \beta_{\sigma}} \left\| |x|^{\beta_x \beta_{\sigma}} S_{\vartheta, \delta, x_0}(v) \Delta K \right\|_{L^2(\Lambda_{\delta, x_0})}^2. \label{eq:integrand Fm}
\end{align}

Without loss of generality assume $\beta_x \beta_{\sigma} \in(0,1)$ and $\beta_x \beta_{\sigma} \neq 1/3$ (otherwise decrease $\beta_x$ or $\beta_\sigma$ slightly). Using Lemma \ref{lemma:JJinequality}, we  obtain from \eqref{eq:integrand Fm}
\begin{align*}
&\int_0^{\delta^{-2} t} \mathbb E \left\| \left( \sigma(X(t - \delta^2 v, \delta \cdot + x_0)) - \sigma(X(t - \delta^2 v, x_0)) \right) S_{\vartheta, \delta, x_0}(v) \Delta K \right\|_{L^2(\Lambda_{\delta, x_0})}^2 \, dv \\
&\lesssim \delta^{2 \beta_x \beta_{\sigma}} \int_0^{\delta^{-2} T} (1+v)^{3(\beta_x \beta_{\sigma}-1)/2} \, dv
\lesssim \delta^{2\beta_x\beta_\sigma}\left(1+\delta^{-2(3\beta_x\beta_\sigma/2-1/2)}\right).
\end{align*}
Since the bound applies uniformly to all $t\in[0,T]$, we deduce from \eqref{eq:Fm}
$$
\varepsilon_{\delta}^{-4}\mathbb E \int_{\delta}^T F_-^2(t) \, dt \lesssim \varepsilon_{\delta}^{-4}\delta^{-2} (\delta^{2\beta_x\beta_\sigma}
+\delta^{1-\beta_x\beta_\sigma})=o(\delta^{-2})
$$
in view of \eqref{eq:epsilon}, which remained to be proved.

(iv).  We obtain, using the Cauchy-Schwarz inequality,
\begin{align}
&\mathbb E \left| \mathcal L_{\delta}^{(iii)} - \mathcal L_{\delta}^{(iv)} \right| \lesssim \varepsilon_{\delta}^{-4} \mathbb E \int_{\delta}^T \Big| \Big( \int_0^t \sigma(X(s, x_0)) \left\langle S_{\vartheta} (t-s) \Delta K_{\delta, x_0}, dW(s) \right\rangle \Big)^2 - \notag \\
&- \Big( \int_{t - \delta}^t \sigma(X(s, x_0)) \left\langle S_{\vartheta} (t-s) \Delta K_{\delta, x_0}, dW(s) \right\rangle \Big)^2 \Big| \, dt \notag \\
&\leq \varepsilon_{\delta}^{-4} \mathbb E \int_{\delta}^T \Big( \int_0^{t - \delta} \sigma(X(s, x_0)) \left\langle S_{\vartheta} (t-s) \Delta K_{\delta, x_0}, dW(s) \right\rangle \Big)^2 \, dt \, + \notag \\
&+2 \varepsilon_{\delta}^{-4} \Big(\mathbb E \int_{\delta}^T \Big( \int_0^{t - \delta} \sigma(X(s, x_0)) \left\langle S_{\vartheta} (t-s) \Delta K_{\delta, x_0}, dW(s) \right\rangle \Big)^2 \, dt\Big)^{1/2} \times  \notag \\
&\times  \Big(\mathbb E \int_{\delta}^T \Big( \int_{t - \delta}^t \sigma(X(s, x_0)) \left\langle S_{\vartheta} (t-s) \Delta K_{\delta, x_0}, dW(s) \right\rangle\Big)^2 \, dt\Big)^{1/2}\notag\\
& =: (I) + (II). \label{Liv}
\end{align}

In the analysis of $(I)$, we use It\^o's isometry, the scaling properties from Lemma \ref{scaling properties} and Lemma \ref{lemma:JJinequality} with $\alpha=0$. We obtain
\begin{align*}
(I) &= \varepsilon_{\delta}^{-4} \int_{\delta}^T \int_{\delta}^{t} \E \left( \sigma^2(X(t - v, x_0)) \right) \| S_{\vartheta} (v) \Delta K_{\delta,x_0} \|_{L^2(\Lambda)}^2 \, dv \, dt \\
&\leq \varepsilon_{\delta}^{-4} \delta^{-2} \overline{\sigma}^2 T \int_{\delta^{-1}}^\infty  \scapro{S_{\vartheta,\delta,x_0} (2v) \Delta K}{\Delta K}_{L^2(\Lambda_{\delta,x_0})} \, dv \\
&= \varepsilon_{\delta}^{-4} \delta^{-2} \overline{\sigma}^2 T (2\theta)^{-1} \scapro{-S_{\vartheta, \delta, x_0} (2\delta^{-1})  K}{\Delta K}_{L^2(\Lambda_{\delta,x_0})}  \\
&\lesssim \varepsilon_{\delta}^{-4} \delta^{-2} (1+2\delta^{-1})^{-3/4}=o(\delta^{-2})
\end{align*}
 due to  \eqref{eq:epsilon}. Given the result for $(I)$, the term $(II)$ is also of order $o(\delta^{-2})$ because the squared second factor is of order $O(\delta^{-2})$:
\begin{align*}
& \mathbb E \int_{\delta}^T \Big( \int_{t - \delta}^t \sigma(X(s, x_0)) \left\langle S_{\vartheta} (t-s) \Delta K_{\delta, x_0}, dW(s) \right\rangle\Big)^2 \, dt\\
&\leq T \bar\sigma^2\int_0^\infty \norm{S_\theta(v)\Delta K_{\delta,x_0}}^2dv \lesssim \delta^{-2},
\end{align*}
using It\^o isometry and arguing as in \eqref{ineq for BDG}.

(v). Denoting
\[
G_{\pm} (t) :=  \int_{t - \delta}^t \left( \sigma(X(s, x_0)) \pm \sigma(X(t - \delta, x_0)) \right) \left\langle S_{\vartheta}(t-s) \Delta K_{\delta, x_0}, dW(s) \right\rangle ,
\]
we compute
\begin{align}
&\mathbb E
| \mathcal L_{\delta}^{(iv)} - \mathcal L_{\delta}^{(v)} | \notag \\
&\lesssim \varepsilon_{\delta}^{-4} \mathbb E \int_{\delta}^T \Big| \Big( \int_{t - \delta}^t \sigma(X(s, x_0)) \left\langle S_{\vartheta} (t-s) \Delta K_{\delta, x_0}, dW(s) \right\rangle \Big)^2 -  \notag \\
&\quad  - \Big( \int_{t - \delta}^t \sigma(X(t - \delta, x_0)) \left\langle S_{\vartheta} (t-s) \Delta K_{\delta, x_0}, dW(s) \right\rangle \Big)^2 \Big| \, dt \notag\\
&\lesssim \varepsilon_{\delta}^{-4} \mathbb E \int_{\delta}^T \abs{G_+(t) G_-(t)} \, dt
\leq \varepsilon_{\delta}^{-4} \Big( \mathbb E \int_{\delta}^T G_+^2(t) \, dt \Big)^{1/2} \Big( \mathbb E \int_{\delta}^T G_-^2(t) \, dt \Big)^{1/2}. \label{Kvi 2}
\end{align}
It\^o's isometry  yields
\begin{align*}
&\mathbb E  G_{\pm}^2(t) =  \mathbb E  \int_0^{\delta} (\sigma(X(t - v, x_0)) \pm \sigma(X(t - \delta, x_0)))^2 \left\| S_{\vartheta} (v) \Delta K_{\delta,x_0} \right\|^2 \, dv.
\end{align*}
Hence, $\mathbb E \int_{\delta}^T G_+^2(t) \, dt = O(\delta^{-2})$ follows from the boundedness of $\sigma$ (i.e., condition \eqref{Eqbarsigma}) and the argument in \eqref{ineq for BDG}.

For the term with $G_-$ we have
from Assumptions \ref{ass:sigma'} and \ref{ass:X'} that
$\mathbb E( \sigma(X(t -  v, x_0)) - \sigma(X(t - \delta, x_0)) )^2\lesssim \delta ^{2 \beta_t \beta_{\sigma}}$ uniformly over $t$ and $v$. The same argument thus gives here $\mathbb E \int_{\delta}^T G_-^2(t) \, dt = O(\delta^{2 \beta_t \beta_{\sigma}-2})$. Insertion into \eqref{Kvi 2} and noting condition \eqref{eq:epsilon} yields the overall rate $o(\delta^{-2})$.

(vi). Introduce for $t\in[\delta,T]$
\begin{align*}
H(t) &= \Big( \int_{t - \delta}^t \left\langle S_{\vartheta} (t-s) \Delta K_{\delta, x_0}, dW(s) \right\rangle \Big)^2, \quad
\tilde{H}(t) = H(t) - \mathbb E H(t)
\end{align*}
and compute
\begin{align}
&\mathbb E \left( \mathcal L_{\delta}^{(v)} - \mathcal L_{\delta}^{(vi)} \right)^2 = \mathbb E \Big( \int_{\delta}^T f_{\delta}(X(t - \delta)) \sigma^2(X(t - \delta, x_0)) \tilde{H}(t) \, dt \Big)^2 \notag \\
&= 2\mathbb E \int_{\delta}^T \int_{u}^T f_{\delta}(X(t - \delta)) f_{\delta}(X(u - \delta)) \times \notag \\
&\times \sigma^2(X(t - \delta, x_0)) \sigma^2(X(u - \delta, x_0)) \tilde{H}(t) \tilde{H}(u) \, dt \, du. \label{Kvii}
\end{align}
We  analyse the integral on two sets: (a) $t > u + \delta$ and (b) $ u + \delta > t > u$.

(a) $t > u + \delta$. We condition on $\mathcal F_{t - \delta}$ and obtain
\begin{align*}
&\int_{\delta}^{T - \delta} \int_{u + \delta}^T \mathbb E \Big( \mathbb E \Big[ f_{\delta}(X(t - \delta)) f_{\delta}(X(u - \delta)) \times  \\
&\times \sigma^2(X(t - \delta, x_0)) \sigma^2(X(u - \delta, x_0)) \tilde{H}(t) \tilde{H}(u) \Big| \mathcal F_{t - \delta} \Big] \Big) \, dt \, du\\
&=\int_{\delta}^{T - \delta} \int_{u + \delta}^T \E \tilde{H}(t) \E \Big( f_{\delta}(X(t - \delta)) f_{\delta}(X(u - \delta)) \times \\
& \times \sigma^2(X(t - \delta, x_0)) \sigma^2(X(u - \delta, x_0)) \tilde{H}(u) \Big) \, dt \, du\\
&=0,
\end{align*}
using that $\tilde{H}(t)$ is independent of $\mathcal F_{t - \delta}$ with $\E \tilde{H}(t) =0$ and that the other factors are $\mathcal F_{t - \delta}$-measurable.

(b) $ u + \delta > t > u$. In this case, we use the upper bounds for $\sigma$ and $f_{\delta}$ and bound \eqref{Kvii} (up to a positive constant) by
\begin{equation} \label{Kvii 2}
\varepsilon_{\delta}^{-8} \int_{\delta}^T \int_u^{\left(u + \delta \right) \wedge T} \mathbb E | \tilde{H}(t) \tilde{H}(u) | \, dt \, du\le \varepsilon_{\delta}^{-8} T \delta \sup_{t\in[\delta,T]} \E \tilde{H}^2(t),
\end{equation}
where the last step involves the Cauchy-Schwarz inequality. The analogous calculations to \eqref{ineq for BDG} yield
\[ \E \tilde{H}^2(t) \leq \E H^2(t) =O(\delta^{-4})\]
uniformly over $t\in[\delta,T]$.
We conclude $(\mathbb E ( \mathcal L_{\delta}^{(v)} - \mathcal L_{\delta}^{(vi)} )^2)^{1/2}\lesssim \eps_\delta^{-4}\delta^{-3/2}$, implying $\abs{\mathcal L_{\delta}^{(v)} - \mathcal L_{\delta}^{(vi)}}=o_{\PP}(\delta^{-2})$ due to \eqref{eq:epsilon}.

(vii). By translating the integrand in $\mathcal L_{\delta}^{(vii)}$ and by It\^o's isometry we have
\begin{align*}
  \mathcal L_{\delta}^{(vi)} - \mathcal L_{\delta}^{(vii)}&=\int_\delta^T f_\delta(X(t-\delta))\sigma^2(X(t-\delta,x_0))\,dt \times\\
  &\quad\times  \Big( \int_0^\delta \norm{S_\theta(v)\Delta K_{\delta, x_0}}^2\,dv-(2\theta)^{-1}\norm{K'}_{L^2(\R)}^2\delta^{-2}\Big)\\
  &\quad -(2\theta)^{-1}\norm{K'}_{L^2(\R)}^2\delta^{-2}\int_{T-\delta}^T f_\delta(X(s))\sigma^2(X(s,x_0))\,ds
\end{align*}
By the uniforms bounds on $f_\delta$ and $\sigma$ the last term is of order $O(\delta^{-2}\delta\eps_\delta^{-4})$.

Next, we compute by the scaling properties in Lemma \ref{scaling properties} and by partial integration
\begin{align*}
\int_0^\delta \norm{S_\theta(v)\Delta K_{\delta, x_0}}^2dv &= \delta^{-2}\int_0^{\delta^{-1}} \scapro{S_{\theta,\delta,x_0}(2v)\Delta K}{\Delta K}_{L^2(\Lambda_{\delta,x_0})}dv\\
&= \delta^{-2}(2\theta)^{-1} \scapro{(S_{\theta,\delta,x_0}(2\delta^{-1})-I)\Delta K}{K}_{L^2(\Lambda_{\delta,x_0})}\\
&=\delta^{-2}(2\theta)^{-1}\left(\norm{K'}_{L^2(\R)}^2+\scapro{S_{\theta,\delta,x_0}(2\delta^{-1})\Delta K}{K}_{L^2(\Lambda_{\delta,x_0})}\right).
\end{align*}
Now, bounding the scalar product by the Cauchy-Schwarz inequality and Lemma \ref{lemma:JJinequality} with $\alpha=0$, we see that it is of order $O(\delta^{3/4})$. Using the upper bounds for $f_\delta$ and $\sigma$ again, we thus obtain
\[\mathcal L_{\delta}^{(vi)} - \mathcal L_{\delta}^{(vii)}=O_{\PP}(\eps_\delta^{-4}\delta^{-2}\delta^{3/4}+\eps_\delta^{-4}\delta^{-2}\delta)=o_{\PP}(\delta^{-2})\]
by \eqref{eq:epsilon}.
\end{proof}

\subsection{Asymptotics of quadratic variations and related terms}

While Proposition \ref{prop:K} has treated the centered process $X_{\delta, x_0}^{\Delta} (t) - \E X_{\delta, x_0}^{\Delta} (t)$, we shall now consider the terms
\begin{align}
\tilde{\mathcal L}_{\delta} &= \int_0^T f_{\delta}(X(t)) \left( X_{\delta, x_0}^{\Delta} (t) \right)^2 \, dt. \label{eq:tildeL}
\end{align}
Having achieved the representation $\mathcal L_{\delta}^{(vii)}$, we are now ready to determine the limit of $\delta^2 \tilde{\mathcal L}_{\delta}$ as $\delta \rightarrow 0$.

\begin{corollary} \label{convergence of K_delta}
Grant Assumptions \ref{ass:sigma} with \eqref{Eqbarsigma}, \ref{ass:X} and let $f_{\delta}$ satisfy condition (F1) or grant Assumptions \ref{ass:sigma'} with \eqref{Eqbarsigma}, \ref{ass:X'} with \eqref{eq:epsilon} and let $f_{\delta}$ satisfy condition (F2). Suppose as $\delta\to 0$
\[ \int_0^T f_\delta(X(t))\sigma^2(X(t,x_0))\,dt\xrightarrow{\PP} \Psi\]
for some random variable $\Psi$.
 Then $\tilde{\mathcal L}_{\delta}$ from \eqref{eq:tildeL} satisfies for $\delta \rightarrow 0$
$$
\delta^2 \tilde{\mathcal L}_{\delta} \stackrel{\mathbb P}{\rightarrow} \frac{\| K' \|_{L^2(\mathbb R)}^2}{2 \vartheta} \Psi.
$$
\end{corollary}

\begin{proof}
The corresponding convergence for $\delta^2{\mathcal L}_{\delta}$, based on the centered process $X_{\delta, x_0}^{\Delta} (t) - \E X_{\delta, x_0}^{\Delta} (t)$, follows directly from part (vii) of Proposition \ref{prop:K}.

For $\tilde{\mathcal L}_{\delta}$ note that
\begin{align}
\delta^2 (\tilde{\mathcal L}_{\delta}-\mathcal L_{\delta})
&= \delta^2 \int_0^T f_{\delta}(X(t)) \left\langle S_{\vartheta} (t) X_0, \Delta K_{\delta, x_0} \right\rangle^2 \, dt \notag \\
&+ 2 \delta^2 \int_0^T f_{\delta}(X(t)) \left(X_{\delta, x_0}^{\Delta} (t)-\E X_{\delta, x_0}^{\Delta} (t) \right) \left\langle S_{\vartheta} (t) X_0, \Delta K_{\delta, x_0} \right\rangle \, dt. \label{delta 2 K}
\end{align}
We apply the upper bound to $|f_{\delta}|$ from condition (F1) (respectively (F2)) to the first term and obtain its convergence to zero by Lemma \ref{lemma:ini} using $\eps_\delta^{-4}\delta^{-11/6}=o(\delta^{-2})$ due to \eqref{eq:epsilon} in the second case.
The second term converges to zero in probability using the Cauchy-Schwarz inequality and $\delta^2 {\mathcal L}_{\delta}=O_P(1)$ from above. This gives the result for $\delta^2\tilde{\mathcal L}_{\delta}$.
\end{proof}

\begin{proposition} \label{convergence of IJ_delta}
The following holds as $\delta \rightarrow 0$:
\begin{itemize}
\item[(i)] under Assumptions \ref{ass:sigma} with \eqref{Eqbarsigma} and \ref{ass:X}:
$$
\delta^2 \mathcal J_{\delta} \stackrel{\mathbb P}{\rightarrow} \frac{\| K' \|_{L^2(\mathbb R)}^2}{2 \vartheta} \int_0^T \sigma^2 (X(t, x_0)) \, dt,
$$
\item[(ii)] under Assumptions \ref{ass:sigma} with \eqref{Eqbarsigma} and \ref{ass:X}:
$$
\delta^2 \mathcal I_{\delta} \stackrel{\mathbb P}{\rightarrow} \frac{\| K' \|_{L^2(\mathbb R)}^2 \| K \|_{L^2(\mathbb R)}^2}{2 \vartheta} \int_0^T \sigma^4 (X(t, x_0)) \, dt,
$$
\item[(iii)] under Assumptions \ref{ass:sigma} with \eqref{Eqbarsigma} and \ref{ass:X} and assuming $\underline{\sigma} > 0$:
$$
\delta^2 \tilde{\mathcal I}_{\delta} \stackrel{\mathbb P}{\rightarrow} \frac{T \| K' \|_{L^2(\mathbb R)}^2}{2 \vartheta \| K \|_{L^2(\mathbb R)}^2},
$$
\item[(iv)] under Assumptions \ref{ass:sigma'}  with \eqref{Eqbarsigma} and \ref{ass:X'}  with \eqref{eq:epsilon}:
$$
\delta^2 \mathcal J_{\delta}^{\star} \stackrel{\mathbb P}{\rightarrow} \frac{T^{\star} \| K' \|_{L^2(\mathbb R)}^2}{2 \vartheta \| K \|_{L^2(\mathbb R)}^2},
$$
\item[(v)] under Assumptions \ref{ass:sigma'} with \eqref{Eqbarsigma} and \ref{ass:X'} with \eqref{eq:epsilon}:
$$
\delta^2 \mathcal I_{\delta}^{\star} \stackrel{\mathbb P}{\rightarrow} \frac{T^{\star} \| K' \|_{L^2(\mathbb R)}^2}{2 \vartheta \| K \|_{L^2(\mathbb R)}^2}.
$$
\end{itemize}
\end{proposition}

\begin{proof}
We apply Proposition \ref{convergence of K_delta} to the functionals proposed in Lemma \ref{lemma:f}.

(i). We use $f_{\delta}(X(t)) = 1$ so that $\Psi=\int_0^T\sigma^2(X(t,x_0))\,dt$ and the result follows.

(ii). We use $f_{\delta}(X(t)) = \| \sigma (X(t)) K_{\delta, x_0} \|^2$.
Since $f_{\delta}(X(t))  \rightarrow \sigma^2(X(t, x_0)) \| K \|_{L^2(\mathbb R)}^2$ by Lemma \ref{lemma: limit sigma K}, the limit $\Psi = \int_0^T\sigma^4(X(t, x_0)) \| K \|_{L^2(\mathbb R)}^2 \, dt$ follows by  dominated convergence  because $\overline{\sigma}^4 \| K \|_{L^2(\mathbb R)}^2$ is an integrable majorant.

(iii). We use $f_{\delta}(X(t)) = \| \sigma (X(t)) K_{\delta, x_0} \|^{-2}$.
Since $f_{\delta}(X(t))\sigma^2(X(t,x_0)) \rightarrow \| K \|_{L^2(\mathbb R)}^{-2}$, we obtain the  limit $\Psi = T\| K \|_{L^2(\mathbb R)}^{-2}$. The integrable majorant for the $dt$--integral over $f_{\delta}(X(t)) \sigma^2(X(t, x_0))$ can be taken as $\frac{\overline{\sigma}^2}{\underline{\sigma}^2 \| K \|_{L^2(\mathbb R)}^2}$  (up to a multiplicative constant).

(iv). We use $f_{\delta}(X(t)) = \frac{1}{\| \sigma (X(t)) K_{\delta, x_0} \|^2 + \varepsilon_{\delta}^2}$.
Since we have only an upper bound to $\sigma(\cdot)$ in this case, obtaining an integrable majorant to
$
\frac{\sigma^2(X(t, x_0))}{\| \sigma (X(t)) K_{\delta, x_0} \|^2 + \varepsilon_{\delta}^2}
$
is not so straightforward. We use $|A^2 - B^2| = |A + B| |A - B|$, the upper bound $\overline{\sigma}$ and Assumptions \ref{ass:sigma'} and \ref{ass:X'}  to obtain
\begin{align}
&\E\abs{\| \sigma (X(t)) K_{\delta, x_0} \|^2 -\sigma^2(X(t, x_0))\norm{K}_{L^2(\R)}^2}\notag\\
 &\le \int_\Lambda \E\abs{\sigma^2(X(t,x))-\sigma^2(X(t,x_0))}K_{\delta,x_0}^2(x) \, dx\notag\\
&\lesssim \int_\Lambda \abs{x-x_0}^{\beta_x\beta_\sigma}K_{\delta,x_0}^2(x) \, dx\notag\\
&= \int_{\R} \abs{\delta y}^{\beta_x\beta_\sigma}K^2(y) \, dy \lesssim \delta^{\beta_x\beta_\sigma}.\label{eq:sigmanorm}
\end{align}
This gives the uniform majorant in $\delta$ and $t$
\begin{align*}
\E\frac{\sigma^2(X(t, x_0))}{\| \sigma (X(t)) K_{\delta, x_0} \|^2 + \varepsilon_{\delta}^2}&\lesssim \E\frac{\| \sigma (X(t)) K_{\delta, x_0} \|^2} {\norm{K}_{L^2(\R)}^2(\| \sigma (X(t)) K_{\delta, x_0} \|^2 + \varepsilon_{\delta}^2)}+\frac{\delta^{\beta_x\beta_\sigma}}{\norm{K}_{L^2(\R)}^2\eps_\delta^2}\\
&\le \frac2{\norm{K}_{L^2(\R)}^2}
\end{align*}
for sufficiently small $\delta$ due to \eqref{eq:epsilon}. Next, we  determine the pointwise limit of the $L^1(\PP)$-distance. In view of \eqref{eq:sigmanorm} and $\delta^{\beta_x\beta_\sigma}\eps_\delta^{-2}\to 0$ due to \eqref{eq:epsilon} it suffices to note for all $t$
\[
\E\babs{\frac{\sigma^2(X(t, x_0))}{\sigma^2(X(t, x_0))\norm{K}_{L^2(\R)}^2 + \varepsilon_{\delta}^2}-\frac1{\norm{K}_{L^2(\R)}^2}{\bf 1}(\sigma^2(X(t,x_0))\not=0)}\underset{\delta \rightarrow 0}{\longrightarrow} 0.
\]
We conclude that the integral converges in $L^1(\PP)$ and thus in probability to
\[\Psi=\int_0^T \frac1{\norm{K}_{L^2(\R)}^2}{\bf 1}(\sigma^2(X(t,x_0))\not=0)\,dt=\norm{K}_{L^2(\R)}^{-2} T^\star.\]

(v). We use $f_{\delta}(X(t)) = \frac{\| \sigma (X(t)) K_{\delta, x_0} \|^2}{(\| \sigma (X(t)) K_{\delta, x_0} \|^2 + \varepsilon_{\delta}^2)^2}$.
An integrable majorant in $t$
is given by  $2 \| K \|_{L^2(\mathbb R)}^{-2}$ and the limit $\Psi=\norm{K}_{L^2(\R)}^{-2} T^\star$ is determined as in the previous case.
\end{proof}

\subsection{Proof of the main theorems} \label{app:stableconv}

\begin{proof}[Proof of Theorem \ref{thm:hatvartheta}]
Assume first in addition the boundedness \eqref{Eqbarsigma} of $\sigma(\cdot)$. Consider the error decomposition \eqref{eq:errorforhat}. The asymptotic properties of $\delta^2 \mathcal I_{\delta}$ and $\delta^2 \mathcal J_{\delta}$ are established in Proposition \ref{convergence of IJ_delta}(i),(ii), therefore the second factor converges in probability to the random variable
$$
\frac{(2 \vartheta)^{1/2} \| K \|_{L^2(\mathbb R)}}{\| K' \|_{L^2(\mathbb R)}} \cdot \frac{(\int_0^T \sigma^4 (X(t, x_0)) \, dt )^{1/2}}{\int_0^T \sigma^2 (X(t, x_0)) \, dt}.
$$

For the first factor, let $Y_{\delta}(t, x) := \delta X_{\delta, x_0}^{\Delta}(t) \sigma(X(t, x)) K_{\delta, x_0}(x)$ and check the two conditions of Proposition \ref{prop:StableCLT}.
Since
$$
\int_0^T \| Y_{\delta}(t) \|^2 \, dt = \delta^2 \mathcal I_{\delta} \stackrel{\mathbb P}{\rightarrow} \frac{\| K' \|_{L^2(\mathbb R)}^2 \| K \|_{L^2(\mathbb R)}^2}{2 \vartheta} \int_0^T \sigma^4(X(t, x_0)) \, dt
$$
by Proposition \ref{convergence of IJ_delta}(ii), condition (C1) is satisfied with $s(t) = \frac{\| K' \|_{L^2(\mathbb R)} \| K \|_{L^2(\mathbb R)}}{(2 \vartheta)^{1/2}} \sigma^2(X(t, x_0))$.
The support condition (C2') follows directly from the definition of $Y_\delta(t)$ as a multiple of $K_{\delta,x_0}$.
Proposition \ref{prop:StableCLT} thus shows $\delta \mathcal M_{\delta} \xrightarrow{stably} \int_0^T s(t) \, dB(t)$ as $\delta \rightarrow 0$ with an independent scalar Brownian motion $B$. On the event $\{ \int_0^T\sigma^2(X(t,x_0))\,dt>0\}$ we infer $\frac{\delta \mathcal M_{\delta}}{\delta \mathcal I_{\delta}^{1/2}} \xrightarrow{stably} Z$ where $Z \sim N(0,1)$ is independent of the $\sigma$--algebra $\scr F_T$. We conclude by applying Slutsky's lemma.

If $\sigma(\cdot)$ is continuous, but unbounded, then consider the events $\Omega_R=\{\max_{t\in[0,T],x\in\bar\Lambda}\abs{X(t,x)}\le R\}$ and set $\sigma_R(y):=\sigma(y)\wedge R$ for $R>0$. By Assumption \ref{ass:X} $X$ is continuous in time and space such that $\PP(\Omega_R)\uparrow 1$ follows for $R\to\infty$. Moreover, by uniqueness $X$ equals on $\Omega_R$ the solution $X_R$ $\PP$--a.s. of the SPDE \eqref{equation SPDE} in terms of $\sigma_R$ and also the estimators $\hat\theta_\delta$, based on observations from $X$ and $X_R$, respectively, agree on $\Omega_R$. The stable limit theorem for $X_R$ with bounded $\sigma_R(\cdot)$ thus yields
\[
\delta^{-1} ( \hat{\vartheta}_{\delta} - \vartheta ){\bf 1}_{\Omega_R} \xrightarrow{stably} \frac{(2 \vartheta)^{1/2} \| K \|_{L^2(\mathbb R)}}{\| K' \|_{L^2(\mathbb R)}} \cdot \frac{(\int_0^T \sigma^4 (X(t, x_0)) \, dt )^{1/2}}{\int_0^T \sigma^2 (X(t, x_0)) \, dt}{\bf 1}_{\Omega_R} \cdot Z.
\]
Since this holds for all $R\in\N$, it is also true for the limit $\lim_{R\to\infty}{\bf 1}_{\Omega_R}={\bf 1}$ $\PP$--a.s.
This gives the proof in the case of unbounded continuous $\sigma$.
\end{proof}

\begin{proof}[Proof of Theorem \ref{thm:tildevartheta}]
Assume first in addition the boundedness \eqref{Eqbarsigma} of $\sigma(\cdot)$. The decomposition of the error follows from \eqref{eq:errorfortilde}:
$$
\delta^{-1} ( \tilde{\vartheta}_{\delta} - \vartheta ) = \frac{\tilde{\mathcal M}_{\delta}}{\tilde{\mathcal I}_{\delta}^{1/2}} \cdot \frac{1}{( \delta^2 \tilde{\mathcal I}_{\delta} )^{1/2}}.
$$

A standard continuous martingale central limit theorem (e.g., \cite{Ku}, Theorem 1.19) provides the convergence of the first factor to $N(0,1)$ in distribution.

Proposition \ref{convergence of IJ_delta}(iii) yields the convergence
$$
\delta^2 \tilde{\mathcal I}_{\delta} \stackrel{\mathbb P}{\rightarrow} \frac{T \| K' \|_{L^2(\mathbb R)}^2}{2 \vartheta \| K \|_{L^2(\mathbb R)}^2}
$$
and the result  follows by applying Slutsky's lemma.

If $\sigma(\cdot)$ is continuous, but unbounded, we use $\Omega_R$, $\sigma_R$ and $X_R$ as in the proof of Theorem \ref{thm:hatvartheta} and write $Z_\delta=\delta^{-1}(\tilde\theta_\delta-\theta)$ and $Z_{\delta,R}=\delta^{-1}(\tilde\theta_{\delta,R}-\theta)$ with the MNE $\tilde \theta_{\delta,R}$  based on observations from $X_R$. We check weak convergence to $Z\sim N(0,2\theta\norm{K}_{L^2(\R)}^2/(T\norm{K'}_{L^2(\R)}^2))$ for test functions $f\in C_b(\R)$ via
\begin{align*}
&\limsup_{\delta\to 0} \abs{\E[f(Z_\delta)]-\E[f(Z)]}\\
 &\le \limsup_{\delta\to 0} \Big(\abs{\E[f(Z_{\delta,R})]-\E[f(Z)]}+\E[\abs{f(Z_{\delta,R})-f(Z_\delta)}{\bf 1}_{\Omega_R^C}]\Big)\\
&\le 0+2\norm{f}_\infty\PP(\Omega_R^C)\text{ for any } R>0.
\end{align*}
Letting $R\to\infty$, the limit is zero and the general result  follows.
\end{proof}

\begin{proof}[Proof of Theorem \ref{thm:starvartheta}]
Exactly as for Theorem \ref{thm:hatvartheta} we argue first for bounded $\sigma(\cdot)$ and then extend to unbounded continuous $\sigma(\cdot)$.
Consider the error decomposition \eqref{eq:errorforstar}. The asymptotic properties of $\delta^2 \mathcal I_{\delta}^{\star}$ and $\delta^2 \mathcal J_{\delta}^{\star}$ are established in Proposition \ref{convergence of IJ_delta}(iv),(v). Therefore the second factor converges in probability to the random variable
$$
\frac{(2 \vartheta)^{1/2} \| K \|_{L^2(\mathbb R)}}{(T^{\star})^{1/2} \| K' \|_{L^2(\mathbb R)}} \text{ on $\{T^\star>0\}$}.
$$

For the first factor, let $Y_{\delta}(t,x) := \frac{\delta X_{\delta, x_0}^{\Delta}(t) \sigma(X(t,x)) K_{\delta, x_0}(x)}{\| \sigma (X(t)) K_{\delta, x_0} \|^2 + \varepsilon_{\delta}^2}$ and check the two conditions of Proposition \ref{prop:StableCLT}.
Since
$$
\int_0^T \| Y_{\delta}(t) \|^2 \, dt = \delta^2 \mathcal I_{\delta}^{\star} \stackrel{\mathbb P}{\rightarrow} \frac{\| K' \|_{L^2(\mathbb R)}^2}{2 \vartheta \| K \|_{L^2(\mathbb R)}^2} \int_0^T \mathbf{1}( \sigma(X(t, x_0)) \neq 0) \, dt
$$
by Proposition \ref{convergence of IJ_delta}(v), condition (C1) is satisfied with $s(t) = \frac{\| K' \|_{L^2(\mathbb R)}}{(2 \vartheta)^{1/2} \| K \|_{L^2(\mathbb R)}} \mathbf{1}(\sigma(X(t, x_0)) \neq 0)$.
The support condition (C2') follows again directly by definition of $Y_\delta(t)$.

Proposition \ref{prop:StableCLT} thus yields $\delta \mathcal M_{\delta}^{\star} \xrightarrow{stably} \int_0^T s(t) \, dB(t)$  with an independent scalar Brownian motion $B$. Hence, $\frac{\delta \mathcal M_{\delta}^{\star}}{\delta (\mathcal I_{\delta}^{\star})^{1/2}} \xrightarrow{stably} Z$ holds on $\{T^\star>0\}$ with $Z \sim N(0,1)$ independent of the $\sigma$--algebra $\scr F_T$. The proof is concluded by applying Slutsky's lemma and approximating possibly unbounded $\sigma$ by the truncated versions $\sigma_R$.
\end{proof}

\section{A stable limit theorem for cylindrical Brownian martingales} \label{sec:stableCLT}

Let $H$ be a separable Hilbert space and $(e_k)_{k \geq 1}$ a complete orthonormal system in $H$. Let $(W_k(t), t \geq 0)_{k \geq 1}$ be a sequence of independent real-valued standard Brownian motions. Then $W(t) = \sum_{k \geq 1} W_k(t) e_k$ is an $H$-valued cylindrical Brownian motion (e.g., Proposition 4.11 in \cite{DPZ}). Consider the filtered pro\-ba\-bi\-li\-ty space $(\Omega, {\scr F}, ({\scr F}_t)_{t \geq 0}, \mathbb P)$, on which $(W_k(t), t \geq 0)_{k \geq 1}$ are defined and where the \textit{Brownian filtration} $({\scr F}_t)_{t \geq 0}$ is the filtration generated by $(W_k(t), t \geq 0)_{k \geq 1}$ and augmented by $\mathbb P$--null sets.

We start with a Hilbert space-valued Brownian martingale representation theorem, which follows by approximation from the finite-dimensional version, but does not seem readily available in the literature.

\begin{proposition} \label{prop:CBL}
Let $(M(t), t \geq 0)$ be a square-integrable real-valued martingale with respect to $({\scr F}_t)_{t \geq 0}$ and with c\`adl\`ag paths, $M(0) = 0$. Then there exist progressively measurable processes $(F_k(t), t \geq 0)_{k \geq 1}$ satisfying $\sum_{k=1}^{\infty} \int_0^T \mathbb E F_k^2 (t) \, dt < \infty$ for all $T>0$ and $\mathbb P$--a.s. (with $L^2(\PP)$-convergence)
$$
M(t) = \sum_{k=1}^\infty \int_0^t F_k(s) \, dW_k(s) = \int_0^t \left\langle F(s), dW(s) \right\rangle, \quad t \geq 0,
$$
where $F(s):=\sum_{k=1}^{\infty} F_k(s) e_k$.
\end{proposition}

\begin{proof}
Define the subfiltrations $({\scr F}_t^{(K)})_{t \geq 0}$ generated by $(W_k(t), t \geq 0)_{1 \leq k \leq K}$ and consider
$$
M^{(K)}(t) := \mathbb E [M(t) \, | \, {\scr F}_t^{(K)}].
$$

By the tower property for $s<t$, we have
$$
\mathbb E [M(t) \, | \, {\scr F}_s^{(K)}] = \mathbb E [ \mathbb E [M(t) \, | \, {\scr F}_s] \, | \, {\scr F}_s^{(K)}] = M^{(K)}(s)
$$
and another application of the tower property yields
$$
\mathbb E [M^{(K)}(t) \, | \, {\scr F}_s^{(K)}] = \mathbb E [M(t) \, | \, {\scr F}_s^{(K)}] = M^{(K)}(s).
$$
We conclude that $(M^{(K)}(t), t \geq 0)$ forms an $L^2(\PP)$-martingale with re\-spect to the $K$-dimensional Brownian filtration $({\scr F}_t^{(K)})_{t \geq 0}$. By standard martingale theory (e.g., Theorem 1.3.13 in \cite{KS}) we may choose a c\`adl\`ag version of $(M^{(K)}(t), t \geq 0)$, which we shall do henceforth.

Theorem 3.4.15 in \cite{KS} therefore shows that there are $(F_k(t), t \geq 0)_{1 \leq k \leq K}$ satisfying $\sum_{k=1}^K \int_0^T \mathbb E F_k^2(t) \, dt < \infty$ for all $T > 0$ and $\mathbb P$--a.s.
$$
M^{(K)}(t)=\sum_{k=1}^K \int_0^t F_k(s) \, dW_k(s), \quad t \geq 0.
$$

The uniqueness result of that theorem also shows that for each $K$ the $F_k$, $k=1,\ldots,K$, can be chosen to not depend on $K$ because by independence of $W_K$ from $(W_k, 1 \leq k \leq K-1)$, $K \geq 2$, we have
$$
\mathbb E \Big[ \int_0^t F_K(s)\, dW_K(s) \, \Big| \, {\scr F}_t^{(K-1)} \Big] = 0.
$$

Since ${\scr F}_t$ is generated by $\bigcup_{K \geq 1}{\scr F}_t^{(K)}$, the $L^2$-martingale convergence theorem gives $\lim_{K \rightarrow \infty} M^{(K)}(t) = M(t)$ in $L^2(\PP)$-convergence for every $t \geq 0$. Hence, also $\sum_{k=1}^K \int_0^t F_k(s) \, dW_k(s)$ converges in $L^2(\PP)$ for $K\rightarrow \infty$. By It\^o's isometry this shows that the $L^2(\PP)$-norms converge: $\sum_{k=1}^\infty \int_0^t \mathbb E F_k^2(s) \,ds < \infty$. The limit $\sum_{k=1}^\infty \int_0^t F_k(s) \, dW_k(s)= \int_0^t \left\langle F(s), dW(s) \right\rangle$ is then well defined as an element of $L^2(\PP)$. Moreover, it equals the limit of $M^{(K)}(t)$ whence
$$
M(t) = \sum_{k=1}^\infty \int_0^t F_k(s) \, dW_k(s) = \int_0^t \left\langle F(s), dW(s) \right\rangle
$$
holds $\mathbb P$--a.s. for each fixed $t \geq 0$. Using the c\`adl\`ag path versions on each side, this entails equality for all $t \geq 0$ with probability one.
\end{proof}

\begin{theorem} \label{thm:stclt}
Let $(Y_\delta(t), t \geq 0)$ for $\delta>0$ be progressively measurable $H$-valued processes on $(\Omega,{\scr F},({\scr F}_t)_{t \geq 0}, \mathbb P)$ with $\int_0^T \| Y_\delta(t) \|^2 \, dt < \infty$ and all $T \geq 0$ (or $T \in [0, T_{max}]$). Assume for all $T$:
\begin{itemize}
\item[(C1)] $\int_0^T \| Y_\delta(t) \|^2 \, dt \stackrel{\mathbb P}{\rightarrow} \int_0^T s^2(t) \, dt$ as $\delta \rightarrow 0$ for some progressively measurable real-valued process $(s(t), t \geq 0)$ with $\int_0^Ts^2(t)\,dt<\infty$,
\item[(C2)] $\int_0^T \left\langle Y_\delta(t), F(t) \right\rangle \, dt \stackrel{\mathbb P}{\rightarrow} 0$ as $\delta \rightarrow 0$ for all  progressively measurable $H$-valued processes $(F(t), t \geq 0)$.
\end{itemize}
Then the following stable limit theorem for stochastic integrals holds:
$$
\Big( \int_0^T \left\langle Y_\delta(t), dW(t) \right\rangle, \, T \geq 0 \Big) \xrightarrow{stably} \Big( \int_0^T s(t) \, dB(t), \, T \geq 0 \Big)
$$
with an independent scalar Brownian motion $(B(t), t \geq 0)$ (on an extension of the original filtered probability space).
\end{theorem}
\begin{proof}
Since $M_\delta(T) := \int_0^T \left\langle Y_\delta(t), dW(t) \right\rangle$ is a continuous martingale with quadratic variation $C_\delta(T)=\int_0^T \| Y_\delta(t) \|^2 \, dt$, we can apply Theorem IX.7.3(b) in \cite{JS} with the trivial processes $Z_t = 0$, $B_t = 0$ (in that Theorem), so that it remains to check for all $T$:
\begin{enumerate}
\item[(i)] $C_\delta(T) \stackrel{\mathbb P}{\rightarrow} \int_0^T s^2(t) \, dt$,
\item[(ii)] $\left\langle M_\delta, N \right\rangle_T \stackrel{\mathbb P}{\rightarrow} 0$ for all bounded c\`adl\`ag-martingales $N$ on $(\Omega,{\scr F},({\scr F}_t)_{t \geq 0}, \mathbb P)$ with $N(0) = 0$.
\end{enumerate}

Condition (i) is satisfied by assumption (C1). For condition (ii) we use the Brownian martingale representation from Proposition \ref{prop:CBL}  to represent $N(T) = \int_0^T \left\langle F(t), dW(t) \right\rangle$ with progressively measurable coordinates $(F_k(t), t \geq 0)_{k \geq 1}$ and $\sum_{k=1}^{\infty} \int_0^T \mathbb E F_k^2(t) \, dt < \infty$ for all $T \geq 0$. Then $\left\langle M_{\delta}, N \right\rangle_T = \int_0^T \left\langle Y_\delta(t), F(t) \right\rangle \, dt$ holds and condition (ii) follows from Assumption (C2).
\end{proof}

\begin{corollary} \label{coroll:supp}
Theorem \ref{thm:stclt} holds for $H=L^2(\Lambda)$ if condition (C2) is replaced by the following support condition:
\begin{itemize}
\item[(C2')] There exist deterministic Borel sets $A(\delta)\subset [0,T] \times \Lambda$ with $\supp(Y_{\delta'}) \subset A(\delta)$ Lebesgue-almost everywhere for all $0<\delta' \leq \delta$ and $\lambda (A(\delta)) \rightarrow 0$ as $\delta \rightarrow 0$, where $\lambda$ denotes Lebesgue measure on $[0,T] \times \Lambda$.
\end{itemize}
\end{corollary}
\begin{proof}
Set $F_{\delta}(t) = F(t) \mathbf 1_{A(\delta)^C}$ for $F(\cdot)$ in condition (C2) of Theorem \ref{thm:stclt}. Then by $\lambda(A(\delta)) \rightarrow 0$ and  dominated convergence, $F_{\delta} \rightarrow F$ holds in $L^2([0,T] \times \Lambda)$ for each $\omega \in \Omega$. Clearly, for each $\delta>0$ the support property gives
$$
\int_0^T \left\langle Y_{\delta'}(t), F_{\delta}(t) \right\rangle \, dt \stackrel{\mathbb P}{\rightarrow} 0 \text{ as } \delta' \rightarrow 0.
$$
By the triangle and the Cauchy-Schwarz inequality and using condition (C1), we obtain
\begin{align*}
&\lim_{\delta' \rightarrow 0} \Big| \int_0^T \left\langle Y_{\delta'}(t), F(t) \right\rangle, dt \Big|  \\
&\leq \liminf_{\delta \rightarrow 0} \lim_{\delta' \rightarrow 0} \Big( \Big| \int_0^T \left\langle Y_{\delta'}(t), F(t)-F_{\delta}(t) \right\rangle \, dt \Big| + \Big| \int_0^T \left\langle Y_{\delta'}(t), F_{\delta}(t) \right\rangle \, dt \Big| \Big) \\
&\leq \lim_{\delta' \rightarrow 0} \Big( \int_0^T \| Y_{\delta'}(t) \|^2 \, dt \Big)^{1/2} \liminf_{\delta \rightarrow 0} \Big( \int_0^T \| F(t) - F_{\delta}(t) \|^2 \, dt \Big)^{1/2} + 0 \\
&= \Big( \int_0^T s^2(t) \, dt \Big)^{1/2} \times 0 = 0
\end{align*}
with convergence in probability.
\end{proof}

\section*{Acknowledgment}
We are grateful to Randolf Altmeyer for very helpful discussions, in particular bringing up the ideas for Proposition \ref{prop:K}. Insightful comments and questions by Gregor Pasemann, Eric Ziebell, Pavel K\v{r}\'i\v{z} and two anonymous referees have lead to several improvements.
This research has been funded by Deutsche Forschungsgemeinschaft (DFG) - SFB1294/2 - 318763901.

%

\end{document}